\documentclass[graybox]{svmult}

\usepackage{mathptmx}       
\usepackage{helvet}         
\usepackage{courier}        
\usepackage{type1cm}        
\usepackage{makeidx}         
\usepackage{multicol}        
\usepackage[bottom]{footmisc}

\makeindex             

\usepackage[utf8]{inputenc}
\usepackage{microtype}
\usepackage{tablefootnote}
\usepackage{placeins}
\usepackage{amssymb,amsfonts,amstext,amsmath}
\usepackage{mathtools}
\usepackage{multirow}
\usepackage[all]{xy}

\newcommand{\ZZ}{\mathbb{Z}}
\newcommand{\CC}{\mathbb{C}}
\newcommand{\QQ}{\mathbb{Q}}
\newcommand{\F}{\mathbb{F}}
\newcommand{\FF}{\F}
\newcommand{\FFbar}{{\overline{\FF}}}
\newcommand{\PP}{\mathbb{P}}
\newcommand{\cE}{\mathcal{E}}
\newcommand{\EC}{\cE}
\newcommand{\cC}{\mathcal{C}}
\newcommand{\C}{\cC}
\newcommand{\JC}{J_{\cC}}
\newcommand{\AV}{\mathcal{A}}
\newcommand{\Mumford}[2]{\ensuremath{\langle{#1},{#2}\rangle}}
\newcommand{\End}{\mathrm{End}}
\newcommand{\Hom}{\mathrm{Hom}}
\newcommand{\Aut}{\mathrm{Aut}}
\newcommand{\GL}{\mathrm{GL}}
\newcommand{\SL}{\mathrm{SL}}
\newcommand{\Sp}{\mathrm{Sp}}
\newcommand{\PGL}{\mathrm{PGL}}
\newcommand{\dualof}[1]{\ensuremath{{#1}^\dagger}}
\newcommand{\subgrp}[1]{\ensuremath{\langle{#1}\rangle}}
\newcommand{\softO}{\ensuremath{\widetilde{O}}}
\newcommand{\frakell}{\mathfrak{l}}

\usepackage{breakurl}
\usepackage{hyperref}

\hypersetup{
    pdftitle={Isogenies for point counting on genus two hyperelliptic
    curves with maximal real multiplication},    
    pdfauthor={Sean Ballentine, Aurore Guillevic, Elisa Lorenzo
      Garc\'ia, Chloe Martindale, Maike Massierer, Benjamin Smith,
      Jaap Top},             
    pdfsubject={},           
    colorlinks=true,         
    linkcolor=black,         
    citecolor=black,         
    filecolor=black,         
    urlcolor=black,          
}

\begin{document}

\title{%
    Isogenies for point counting 
    \texorpdfstring{\\}{ }
    on genus two hyperelliptic curves
    \texorpdfstring{\\}{ }
    with maximal real multiplication}
\titlerunning{Genus-2 point counting with maximal RM}

\authorrunning{%
    Ballentine,
    Guillevic,
    Lorenzo Garc\'ia,
    Martindale,
    Massierer,
    Smith,
    and 
    Top
}
\author{%
    Sean Ballentine,
    Aurore Guillevic,
    Elisa Lorenzo Garc\'ia,
    Chloe Martindale,
    \texorpdfstring{\\}{ }
    Maike Massierer,
    Benjamin Smith,
    and 
    Jaap Top
}
\institute{
    Sean Ballentine
    \at
    Department of Mathematics,
    University of Maryland,
    4176 Campus Dr.,
    College Park, MD 20742-4015,
    USA
    \email{seanbal@math.umd.edu}
    \and
    Aurore Guillevic
    \at
    Inria Nancy Grand Est, Equipe Caramba, 615 rue du jardin botanique, CS 20101, 54603 Villers-l\`{e}s-Nancy Cedex, France
    \email{aurore.guillevic@inria.fr}
    \and
    Elisa Lorenzo Garc\'ia
    \at
    IRMAR, Universit\'e de Rennes 1, Campus de Beaulieu, 35042  Rennes Cedex, France
    \email{elisa.lorenzogarcia@univ-rennes1.fr}
    \and
    Chloe Martindale
    \at
    Mathematisch Instituut, Universiteit Leiden, P.O. Box 9512,
    2300 RA Leiden, The Netherlands
    \email{chloemartindale@gmail.com}
    \and
    Maike Massierer
    \at
    School of Mathematics and Statistics, University of New South Wales, Sydney NSW 2052, Australia
    \email{maike@unsw.edu.au}
    \and
    Benjamin Smith
    \at
    INRIA and Laboratoire d'Informatique de
    l'\'Ecole polytechnique (LIX), 91120 Palaiseau, France
    \email{smith@lix.polytechnique.fr}
    \and
    Jaap Top
    \at
    University of Groningen,
    Johann Bernoulli Institute for Mathematics and Computer Science,
    P.O. Box 407, 9700AK Groningen, The Netherlands
    \email{j.top@rug.nl}
}

\maketitle

\abstract{%
    Schoof's classic algorithm
    allows point-counting for elliptic curves over finite fields
    in polynomial time.
    This algorithm was subsequently improved by Atkin,
    using factorizations of modular polynomials,
    and by Elkies,
    using a theory of explicit isogenies.
    Moving to Jacobians of genus-2 curves,
    the current state of the art for point counting
    is a generalization of Schoof's algorithm.
    While we are currently missing the tools we need 
    to generalize Elkies' methods
    to genus 2,
    recently Martindale and Milio 
    have computed analogues of modular polynomials
    for genus-2 curves whose Jacobians have real multiplication
    by maximal orders of small discriminant.
    In this article,
    we prove Atkin-style results for genus-2 Jacobians with real
    multiplication by maximal orders,
    with a view to using these new modular polynomials
    to improve the practicality of point-counting
    algorithms for these curves.
}

\abstract*{%
    Schoof's classic algorithm
    allows point-counting for elliptic curves over finite fields
    in polynomial time.
    This algorithm was subsequently improved by Atkin,
    using factorizations of modular polynomials,
    and by Elkies,
    using a theory of explicit isogenies.
    Moving to Jacobians of genus-2 curves,
    the current state of the art for point counting
    is a generalization of Schoof's algorithm.
    While we are currently missing the tools we need 
    to generalize Elkies' methods
    to genus 2,
    recently Martindale and Milio 
    have computed analogues of modular polynomials
    for genus-2 curves whose Jacobians have real multiplication
    by maximal orders of small discriminant.
    In this article,
    we prove Atkin-style results for genus-2 Jacobians with real
    multiplication by maximal orders,
    with a view to using these new modular polynomials
    to improve the practicality of point-counting
    algorithms for these curves.
}

\section{
     Introduction
}

Efficiently computing the number of points
on the Jacobian of a genus 2 curve 
over a finite field
is an important problem in experimental number theory
and number-theoretic cryptography.
When the characteristic of the finite field is small,
Kedlaya's algorithm and its descendants provide an efficient
solution (see \cite{Kedlaya01}, \cite{Harvey07}, and \cite{Harrison12}),
while in extremely small characteristic we have extremely fast
AGM-style algorithms (see for example~\cite{Mestre01}, \cite{Mestre02}, and \cite{Carls04}).
However, the running times of these algorithms are exponential in the
size of the field characteristic; the hardest case,
therefore (and also the most important case for contemporary
cryptographic applications) is where the characteristic is large,
or even where the field is a prime field.

So let \(q\) be a power of a large prime \(p\), 
and let \(\C\) be a genus-2 curve over \(\FF_q\).
Our fundamental problem is to compute the number of \(\FF_q\)-rational
points on the Jacobian \(\JC\) of \(\C\). 

\subsection{The state of the art}

In theory, the problem is solved:
we can compute \(\#\JC(\FF_q)\) in polynomial time
(that is, polynomial in \(\log q\))
using Pila's algorithm~\cite{Pila90},
which is the immediate generalization of 
Schoof's elliptic-curve point-counting algorithm~\cite{Schoof85}
to higher-dimensional abelian varieties.
But the exponent in Pila's polynomial time
is extremely large;
so, despite its theoretical importance,
this algorithm is completely impractical (see~\S\ref{sec:Pila}).
Indeed, to our knowledge it has never been implemented.

Gaudry and Schost have developed and successfully implemented
a much more practical variant of Pila's algorithm
for the case $q=p$ that runs in time \(\softO(\log^8p)\);
not just polynomial time, but on the edge of practicality~\cite{GauSch12}. 
Still, their algorithm requires an extremely intensive
calculation for cryptographic-sized Jacobians:
Gaudry and Schost estimated a running time of 
around one core-month (in 2008)
to compute \(\#\JC(\FF_p)\)
when \(p\) has around 128 bits~\cite{GauSch08}.

The situation improves dramatically if \(\JC\)
is equipped with an efficiently computable \emph{real multiplication} endomorphism.
For such Jacobians, Gaudry, Kohel, and Smith~\cite{Gaudry--Kohel--Smith} 
give an algorithm to compute \(\#\JC(\FF_q)\)
in time \(\softO(\log^5q)\).
This allowed the computation of \(\#\JC(\FF_p)\)
for one curve \(\C\) 
drawn from the genus-2 family in~\cite{TauTopVer91}
with \(p = 2^{512}+1273\)
in about 80 core-days (in 2011); 
this remains, to date, the record for genus-2 point counting 
over prime fields.
For 128-bit fields, the cost is reduced to 3 core hours (in 2011).

All of these algorithms are generalizations of Schoof's algorithm,
which computes the Frobenius trace (and hence the order
\(\#E(\FF_q)\)) of an elliptic curve \(E/\FF_q\) modulo~\(\ell\) 
for a series of small primes~\(\ell\)
by considering the action of Frobenius on the \(\ell\)-torsion.
But Schoof's algorithm is not the state of the art for
elliptic-curve point counting: it has evolved into the
much faster Schoof--Elkies--Atkin (SEA) algorithm, surveyed in~\cite{Schoof95}.
Atkin's improvements involve factoring the \(\ell\)-th modular polynomial 
(evaluated at the \(j\)-invariant of the target curve) 
to deduce information on the Galois structure of the \(\ell\)-torsion, 
which then restricts the possible values of the trace modulo \(\ell\)
(see~\S\ref{sec:Atkin}).
Elkies' improvements involve computing the kernel of a rational
\(\ell\)-isogeny, which takes the place of the full \(\ell\)-torsion;
deducing the existence of the isogeny, and computing its kernel,
requires finding a root of the \(\ell\)-th modular polynomial
evaluated at the \(j\)-invariant of the target curve
(see~\S\ref{sec:Elkies}).

\subsection{Our contributions, and beyond}

Our ultimate goal is to 
generalize Atkin's and Elkies' improvements to genus 2.
In this article, we concentrate on generalizing Atkin's methods
to genus-2 Jacobians with known real multiplication.
This project is prompted by the recent appearance of
two new algorithms for computing modular ideals,
the genus-2 analogue of modular polynomials:
Milio~\cite{Milio15} has computed modular ideals 
for general genus-2 Jacobians,
while Milio~\cite[\S5]{Milio-thesis} and Martindale~\cite{Mar16} 
have independently computed modular ideals for genus-2 Jacobians 
with RM by orders of small discriminants.

To extend Elkies' methods to genus 2 
we would need an analogue of Elkies' algorithm~\cite[\S\S7-8]{Schoof95},
which computes defining equations for the kernel of an isogeny of
elliptic curves (and the isogeny itself) 
corresponding to a root of the evaluated modular polynomial.
We do not know of any such algorithm in genus 2.
Couveignes and Ezome have recently developed an algorithm
to compute explicit \((\ell,\ell)\)-isogenies of genus-2 Jacobians~\cite{Couveignes--Ezome},
presuming that the kernel has already been constructed somehow---%
but kernel construction is precisely the missing step that we need.\footnote{
    We would also like mention Bisson, Cosset, and Robert's 
    \texttt{AVIsogenies} software package~\cite{AVIsogenies},
    which provides some functionality in this direction.
    However, their methods apply to abelian surfaces with a lot of rational 2- and 4-torsion,
    and applying them to general genus-2 Jacobians (with or without known RM)
    generally requires a substantial extension of the base
    field to make that torsion rational. 
    This is counterproductive in the context of point counting.
}

In contrast, Atkin's improvements for elliptic-curve Schoof 
require nothing beyond the modular polynomial itself;
so we can hope to achieve something immediately in genus 2 
by generalizing Atkin's results on factorizations of modular polynomials 
to the decomposition of genus-2 modular ideals.
This is precisely what we do in this article.

We focus on the RM case for three reasons.
First, the construction of explicit modular ideals is furthest advanced
in this case: Milio has constructed modular ideals for primes in
\(\QQ(\sqrt{5})\) of norm
up to 31, while for general Jacobians the current
limit is 3.
It is therefore already possible to compute nontrivial and interesting
examples in the RM case.
Second, point counting is currently much more efficient for
Jacobians with efficiently computable RM;
we hope that, at some point, our methods
can help tip RM point counting from ``feasible'' into ``routine''.
Third, from a purely theoretical point of view, 
the RM case is more similar to the elliptic curve case
in the sense that real multiplication allows us,
in favorable circumstances, to split $\ell$-torsion subgroups of the
Jacobian into groups of the same size as encountered for
elliptic curves.

After recalling the SEA algorithm for elliptic curves in~\S\ref{sec:SEA},
we describe the current state of genus~2 point counting,
and set out our program for a generalized SEA algorithm
in~\S\ref{sec:genus-2-point-counting}.
We describe the modular invariants we need for this
in~\S\ref{sec:invariants},
and the modular ideals that relate them in~\S\ref{sec:modular-ideals}.
We can then state and prove our main theoretical results,
which are generalizations of Atkin's
theorems for these modular ideals, in~\S\ref{sec:Atkin-g2}.
In~\S\ref{sec:sqrt5} we provide some concrete details on 
the special case of RM by \(\QQ(\sqrt{5})\),
before concluding with some experimental results in~\S\ref{sec:experiments}.

\subsection{Vanilla abelian varieties}
We can substantially simplify the task ahead
by restricting our attention to a class of elliptic curves and Jacobians
(more generally, abelian varieties)
with sufficiently general CM endomorphism rings.
The following definition makes this precise.

\begin{definition}
    \label{def:vanilla}
    We say that a \(g\)-dimensional abelian variety \(\AV/\FF_q\)
    is \emph{vanilla}\footnote{%
        Vanilla is the most common and least complicated flavour of abelian
        varieties over finite fields.
        Heuristically, over large finite fields,
        randomly sampled abelian varieties are vanilla
        with overwhelming probability.
        Indeed, being vanilla is invariant in isogeny classes,
        and Howe and Zhu have shown in~\cite[Theorem 2]{Howe--Zhu}
        that the fraction of isogeny classes of \(g\)-dimensional
        abelian varieties over \(\FF_q\) that are ordinary and
        absolutely simple tends to 1 as \(q \to \infty\).
        All absolutely simple ordinary abelian varieties are vanilla,
		except those whose endomorphism algebras contain roots of unity;
        but the number of such isogeny classes for fixed \(g\)
        is asymptotically negligible.
    }
    if its endomorphism algebra \(\End_{\FFbar_q}(\AV)\otimes\QQ\)
    (over the algebraic closure)
    is a CM field of degree \(2g\)
    that does \emph{not} contain any roots of unity other than
    \(\pm1\).
\end{definition}

If an elliptic curve \(\EC/\FF_q\) is vanilla,
then \(\EC\) is nonsupersingular and \(j(\EC)\) is neither \(0\) nor
\(1728\):
these are the conditions Schoof applies systematically
in~\cite{Schoof95}.
We note that in particular,
vanilla abelian varieties are absolutely simple.

To fix notation,
we recall that if \(\AV\) is an abelian variety,
then a \emph{principal polarization}
is an isomorphism \(\xi \colon \AV \to \AV^\vee\) 
associated with an ample divisor class on~\(\AV\),
where \(\AV^\vee = \mathrm{Pic}^0(\AV)\) is the dual abelian variety
(see eg.~\cite[\S13]{Milne}).
We will be working with elliptic curves and Jacobians of genus-2 curves;
these all have a canonical principal polarization.
Each endomorphism \(\phi\) of \(\AV\) 
has a corresponding dual endomorphism \(\phi^\vee\) of \(\AV^\vee\).
If \((\AV,\xi)\) is a principally polarized abelian variety,
then \(\xi\) induces a \emph{Rosati involution} on \(\End(\AV)\),
defined by
\[
    \phi \longmapsto \dualof{\phi} := \xi^{-1}\circ\phi^\vee\circ\xi
    \quad
    \text{for }
    \phi \in \End(\AV)
    \ .
\]
In the world of elliptic curves, the Rosati involution is the familiar
dual.
For vanilla abelian varieties,
the Rosati involution acts as complex conjugation on the endomorphism
ring.

Fix a real quadratic field \(F=\QQ(\sqrt{\Delta})\),
with fundamental discriminant $\Delta>0$ and ring of integers $\mathcal{O}_F$.
We write \(\alpha \mapsto \bar\alpha\) 
for the involution of \(F\) over \(\QQ\);
we emphasize that in this article, 
\(\bar\cdot\) does \emph{not} denote complex conjugation.

From a theoretical point of view,
when talking about real multiplication,
our fundamental data are triples
$(\AV,\xi,\iota)$ where
$\AV$ is an abelian surface,
$\xi\colon\AV\to\AV^{\vee}$ is a principal polarization, 
and $\iota\colon\mathcal{O}_F\hookrightarrow \End(\AV)$ 
is an embedding stable under the Rosati involution
(that is, \(\dualof{\iota(\mu)} = \iota(\mu)\) 
for all \(\mu\) in~\(\mathcal{O}_F\); 
we can then think of the Rosati involution as 
complex conjugation on the endomorphism ring).
While this notation \((\AV,\xi,\iota)\)
may seem quite heavy at first glance,
we remind the reader that generally
there are only two choices of embedding~\(\iota\)
(corresponding to the two square roots of \(\Delta\)),
and we are only really interested in the case where
\(\AV\) is a Jacobian,
in which case the polarization \(\xi\) is canonically determined.

\section{
    Genus one curves: elliptic curve point counting
}
\label{sec:SEA}

We begin by briefly recalling the SEA algorithm for elliptic curve
point counting in large characteristic.
First we describe Schoof's original algorithm~\cite{Schoof95},
before outlining the improvements of Elkies and Atkin.
This will provide a point of reference for comparisons
with genus-2 algorithms.
 
Let \(\EC\) be an elliptic curve over a finite field \(\FF_q\)
of large characteristic (or at least, with \(\mathrm{char}(\FF_q) \gg
\log q\)).
We may suppose that \(\EC\)
is defined by a (short) Weierstrass equation
$\EC: y^2 = x^3 + ax + b$,
with \(a\) and \(b\) in \(\FF_q\).

Like all modern point-counting algorithms,
the Schoof and SEA algorithms compute the characteristic polynomial 
$$ 
    \chi_{\pi}(X) = X^2 - tX + q
$$
of the Frobenius endomorphism \(\pi\) of \(\EC\).
We call $t$ the \emph{trace} of Frobenius.
Since the \(\FF_q\)-rational points on \(\EC\) 
are precisely the fixed points of \(\pi\),
we have 
\[
    \#\EC(\FF_q) = \chi_{\pi}(1) = q + 1 - t
    \ ;
\]
so determining \(\#\EC(\FF_q)\) is equivalent to determining \(t\).
Hasse's theorem tells us that
\begin{equation}
    \label{eq:Hasse}
    |t| \le 2\sqrt{q}
    \ .
\end{equation}

\subsection{Schoof's algorithm}
Schoof's basic strategy is
to choose a set \(\mathcal{L}\) 
of primes \(\ell \not= p\)
such that \(\prod_{\ell\in\mathcal{L}}\ell > 4\sqrt{q}\).
We then compute \(t_\ell := t \bmod{\ell}\)
for each of the primes \(\ell\) in \(\mathcal{L}\),
and then recover the value of $t$
from \(\{(t_\ell,\ell): \ell \in \mathcal{L}\}\)
using the Chinese Remainder Theorem.
The condition \(\prod_{\ell\in\mathcal{L}}\ell > 4\sqrt{q}\)
ensures that \(t\) is completely determined by the collection of \(t_\ell\)
(by Hasse's theorem, Equation~\eqref{eq:Hasse}).

For Schoof's original algorithm,
the natural choice is to let \(\mathcal{L}\)
be the set of the first \(O(\log q)\) primes,
stopping when the condition 
\(\prod_{\ell\in\mathcal{L}}\ell > 4\sqrt{q}\)
is satisfied.
When applying Elkies' and Atkin's modifications,
we will need to be more subtle with our choice of \(\mathcal{L}\).
It is also possible to replace primes with small prime
powers; we will not explore this option here.
 
Now, let \(\ell\) be one of our primes in \(\mathcal{L}\);
our aim is to compute \(t_\ell\).
We know that \(\pi^2(P) - [t]\pi(P) + [q]P = 0\)
for all \(P\) in \(\EC\),
and hence
\[
    \pi^2(P) - [t_\ell]\pi(P) + [q\bmod{\ell}]P = 0
    \quad
    \text{for all } P \in \EC[\ell]
    \ .
\]
We can therefore compute \(t_\ell\)
as follows:
\begin{enumerate}
    \item
        Construct a point \(P\) of order \(\ell\).
    \item
        Compute \(Q = \pi(P)\)
        and \(R = \pi^2(P) + [q\bmod{\ell}]P\).
    \item
        Search for \(0 \le t_\ell < \ell\)
        such that \([t_\ell]Q = R\),
        using Shanks' baby-step giant-step algorithm 
        in the cyclic subgroup of the \(\ell\)-torsion
        generated by \(Q\). 
\end{enumerate}

To construct such a \(P\),
we begin by computing the \(\ell\)-th division polynomial
\(\Psi_\ell\) in \(\FF_q[X]\),
which is the polynomial whose roots
in \(\FFbar_q\) are precisely the \(x\)-coordinates
of the nontrivial points in \(\EC[\ell]\).
When \(\ell\) is odd and prime to \(q\),
we have $\deg \Psi_\ell = (\ell^2-1)/2$.
We then define the ring 
\(A = \FF_q[X,Y]/(\Psi_\ell(X),Y^2 - X^3 - aX - b)\),
and take \(P = (X,Y)\) in \(\EC(A)\).

In order to work efficiently with 
\(Q = \pi(P) = (X^q,Y^q)\)
in the search for \(t_\ell\),
we need to compute a compact form for \(Q\).
This means computing reduced representatives for \(X^q\) and \(Y^q\)
in the ring \(A\)---%
that is,
reducing \(X^q\) modulo \(\Psi_\ell(X)\) and \(Y^q\) modulo
\((\Psi_\ell(X),Y^2-X^3-aX-b)\)---%
which costs $O(\log q)$ \(\FF_q\)-operations.

Having computed \(t_\ell\) for each \(\ell\) in \(\mathcal{L}\),
we recover \(t\) (and hence \(\chi_\pi\))
using the Chinese Remainder Theorem;
this then yields \(\#\EC(\FF_q) = q + 1 - t\).
In cryptographic contexts, 
we are generally interested in curves of (almost) prime order.
One particularly convenient feature of Schoof's algorithm 
is that it allows us to detect small prime factors of
\(\#\EC(\FF_q)\) early:
we can determine if any \(\ell\) in \(\mathcal{L}\) divides \(\#\EC(\FF_q)\)
by 
checking whether \(t_\ell \equiv q + 1 \pmod{\ell}\).
If we find such a factor, 
then we can immediately abort the calculation of~\(t\)
and move on to another candidate curve.

The cost to compute $\chi_\ell$ is $\softO(\ell^2 + (\log q)\ell^2 +
\sqrt{\ell} \ell^2)$ $\F_q$-operations. 
We can take~\(\mathcal{L}\) to be a set of \(O(\log q)\)
primes, the largest of which is in \(O(\log q)\);
the total cost is therefore \(\softO(\log^4q)\)
\(\FF_q\)-operations.

\subsection{Frobenius eigenvalues and subgroups}

Fix a basis of \(\EC[\ell]\),
and thus an isomorphism \(\EC[\ell] \cong {\FF_\ell}^2\).
Now \(\pi\) acts on \(\EC[\ell]\) as an element of \(\GL_2(\FF_\ell)\).
The local characteristic polynomial \(\chi_\ell\)
is just the characteristic polynomial of this matrix.

Likewise, \(\pi\) permutes the \(\ell\)-subgroups of \(\EC[\ell]\);
that is, the one-dimensional subspaces of \(\EC[\ell] \cong {\FF_\ell}^2\).
These are the points of \(\PP(\EC[\ell]) \cong \PP^1(\FF_\ell)\),
and we can consider the image of \(\pi\) in
\(\PGL_2(\FF_\ell) \cong \Aut(\PP(\EC[\ell]))\).
The order of \(\pi\)
as an element of \(\PGL_2(\FF_\ell)\)
is clearly independent of the choice of basis.
 
\begin{proposition}
    \label{prop:EC-trace-relation}
    Let \(\EC/\FF_q\) be an elliptic curve
    with Frobenius endomorphism~\(\pi\),
    and let \(\ell\not=p=\mathrm{char}(\FF_q)\) be an odd prime.
    If \(e\) is the order of the image of \(\pi\) 
    in \(\PGL_2(\FF_\ell)\),
    then the trace \(t\) of \(\pi\) satisfies
    \[
        t^2 = \eta_{e}q 
        \quad 
        \text{in } 
        \FF_\ell 
        \ ,
    \]
    where 
    \(
        \eta_{e} 
        = 
        \begin{cases}
            \zeta + \zeta^{-1}+2
            \text{ with } \zeta\in \FF_{\ell^2}^{\times} 
            \text{ of order } e
            &
            \text{if } \gcd(\ell,e)=1
            \ ,
            \\
            4 & \text{otherwise} \ .
        \end{cases}
    \)
\end{proposition}
\begin{proof}
    We follow the proof of~\cite[Proposition 6.2]{Schoof95}
    (correcting the minor error that leads in the case \(e\) even to an \(e/2\)-th rather than
    \(e\)-th root of unity appearing in the last part of the statement).
    Let \(\lambda_1,\lambda_2\in\FF_{\ell^2}\) be 
    the eigenvalues of the image of \(\pi\) 
    in \(\Aut(\EC[\ell]) \cong \GL_2(\FF_\ell)\);
    then 
    \[
        \lambda_1 + \lambda_2 = t
        \quad
        \text{and}
        \quad
        \lambda_1\lambda_2 = q
        \quad
        \text{in}
        \quad 
        \FF_\ell
        \ .
    \]
In case \(\lambda_1=\lambda_2\) we have \(e\mid\ell\)
and the assertion follows. In case \(\lambda_1\neq\lambda_2\) the given \(e\) is the minimal integer \(>0\) with \(\lambda_1^e = \lambda_2^e\).
In particular \(\gcd(e,\ell)=1\) and \(\lambda_2=\lambda_1\zeta\) for
    some primitive $e$=th root of unity \(\zeta\) (in \(\FF_{\ell^2}\); in fact \(e\mid \ell-1\) in case the eigenvalues are in \(\FF_\ell\) and
\(e\mid \ell+1\) otherwise).
    Hence \(q = \lambda_1\lambda_2 = \lambda_1^{2}\zeta\)
    which implies
    \[
        t^2 = (\lambda_1 + \lambda_2)^2 
            = \lambda_1^2(1+\zeta)^2
            = q\zeta^{-1}(\zeta^2 + 2\zeta + 1)
            = (\zeta + \zeta^{-1}+2)q.
    \]
\end{proof}

\subsection{Modular polynomials and isogenies}

The order-\(\ell\) subgroups of \(\EC[\ell]\)
are precisely the kernels of \(\ell\)-isogenies 
from \(\EC\) to other elliptic curves,
and the set of all such \(\ell\)-isogenies (up to isomorphism)
corresponds to the set of roots of \(\Phi_\ell(j(\EC),x)\)
in \(\FFbar_q\).
The classical modular polynomial $\Phi_\ell(X,Y)$,
of degree $\ell+1$ (in \(X\) and \(Y\)) over \(\ZZ\),
is defined by the property that $\Phi_\ell(j(\cE_1), j(\cE_2)) = 0$ 
precisely when there exists an $\ell$-isogeny $\cE_1 \rightarrow \cE_2$.
For \(\ell\) in \(O(\log q)\),
one can compute $\Phi_\ell(j(\EC),x)$ 
in $\softO(\ell^3)$ $\F_q$-operations
using Sutherland's algorithm~\cite{Sutherland13}.
Alternatively, 
we can use precomputed databases of modular polynomials over \(\ZZ\), 
reducing them modulo \(p\) and specializing them at~\(j(\EC)\).

The Galois orbits of the roots of \(\Phi_\ell(j(\EC),x)\)
correspond to orbits of \(\ell\)-isogeny kernels under \(\pi\),
and to orbits of points of \(\PP^1(\FF_\ell)\) under the image of
\(\pi\) in \(\PGL_2(\FF_\ell)\).
If \(j(\cE_1)\) and \(j(\cE_2)\) are both in \(\FF_{q^k}\),
then the isogeny is defined over \(\FF_{q^k}\) (up to a possible twist);
in particular, its kernel is defined over \(\FF_{q^k}\).
More precisely, we have the following key lemma:

\begin{lemma}[Proposition~6.1 of~\cite{Schoof95}]
    \label{lemma:Schoof-6-1-1}
    Let \(\EC/\FF_q\) be a vanilla elliptic curve
    with Frobenius endomorphism \(\pi\).
    \begin{enumerate}
        \item
            The polynomial \(\Phi_\ell(j(\EC),x)\)
            has a root in \(\FF_{q^e}\)
            if and only if 
            the kernel of the corresponding \(\ell\)-isogeny 
            is a one-dimensional eigenspace of \(\pi^e\) in \(\EC[\ell]\).
        \item
            The polynomial \(\Phi_\ell(j(\EC),x)\)
            splits completely over \(\FF_{q^d}\)
            if and only if \(\pi^d\) acts as a scalar matrix on \(\EC[\ell]\);
            that is,
            if and only if~\(d\) is a multiple of
            the order~\(e\) of the image of \(\pi\) in~\(\PGL_2(\FF_\ell)\).
            In particular, the minimal such~\(d\) is~\(e\).
    \end{enumerate}
\end{lemma}

\subsection{Elkies, Atkin, and volcanic primes}
\label{sec:prime-types}

The primes $\ell \not= p$ 
are divided into 3 classes, or types,
with respect to a given~\(\EC/\FF_q\):
\emph{Elkies}, \emph{Atkin}, and \emph{volcanic}.
The type of \(\ell\) simultaneously reflects
the factorization of \(\Phi_\ell(j(\EC),x)\)
and
the Galois structure of the \(\ell\)-subgroups of \(\EC[\ell]\).
Here we recall a number of facts about these classes,
all of which are proven in~\cite[\S6]{Schoof95};
see also \cite[\S12.4]{Was}.

A prime \(\ell\) is \textbf{Elkies} 
if the ideal \((\ell)\) is split in \(\ZZ[\pi]\);
or, equivalently,
if \(t^2-4q\) is a nonzero square modulo \(\ell\).
Each of the two prime ideals over \((\ell)\)
defines the kernel of an \(\ell\)-isogeny,
\(\phi_i \colon \EC \to \EC_i\) for \(i = 1,2\), say.
This means that \(j(\EC_1)\) and \(j(\EC_2)\)
must be roots in \(\FF_q\) of \(\Phi_\ell(j(\EC),x)\).
Lemma~\ref{lemma:Schoof-6-1-1} then implies that
\begin{equation}
    \label{eq:EC-Elkies-factorization}
    \Phi_\ell(j(\cE),x) 
    = 
    (x-j(\EC_1))(x-j(\EC_2)) \prod_{i=1}^{(\ell-1)/e}f_i(x)
\end{equation}
where each of the \(f_i\) are irreducible of degree \(e\),
and \(e > 1\) is 
the order of the image of \(\pi\) in \(\PGL_2(\FF_\ell)\),
which must divide \(\ell-1\) in this case.

A prime \(\ell\) is \textbf{Atkin} 
if the ideal \((\ell)\) is inert in \(\ZZ[\pi]\);
or, equivalently,
if \(t^2-4q\) is \emph{not} a square modulo~\(\ell\).
There are \emph{no} \(\FF_q\)-rational \(\ell\)-isogenies
from \(\EC\),
and no \(\FF_q\)-rational \(\ell\)-subgroups of \(\EC[\ell]\).
Looking at the modular polynomial,
Lemma~\ref{lemma:Schoof-6-1-1} implies
\begin{equation}
    \label{eq:EC-Atkin-factorization}
    \Phi_\ell(j(\cE),x) = \prod_{i=1}^{(\ell+1)/e}f_i(x)
    \ ,
\end{equation}
where each of the \(f_i\) is an irreducible polynomial of degree \(e\),
and \(e > 1\) is 
the order of the image of \(\pi\) in \(\PGL_2(\FF_\ell)\),
which must divide \(\ell+1\) in this case.

Finally, a prime \(\ell\) is \textbf{volcanic} if
the ideal \((\ell)\) is ramified in \(\ZZ[\pi]\);
or, equivalently,
if \(\ell\) divides \(t^2 - 4q\).
Applying Lemma~\ref{lemma:Schoof-6-1-1},
either
\begin{equation}
    \label{eq:EC-upper-volcanic-factorization}
    \Phi_\ell(j(\cE),x) 
    = 
    \prod_{i=1}^{\ell+1}(x - j_i)  
\end{equation}
with all of the \(j_i\) in \(\FF_q\) (so there are \(\ell+1\) rational \(\ell\)-isogenies, 
and \(\ell+1\) rational \(\ell\)-subgroups of \(\EC[\ell]\));
or
\begin{equation}
    \label{eq:EC-floor-volcanic-factorization}
    \Phi_\ell(j(\cE),x) 
    = 
    (Y-j_1)\cdot f(x)
    \ ,
\end{equation}
with \(f\) irreducible of degree \(\ell\)
(so there is a single rational \(\ell\)-isogeny,
and one rational \(\ell\)-subgroup of \(\EC[\ell]\)).
In either situation,
\(\pi|_{\EC[\ell]}\) acts on \(\EC[\ell]\)
with eigenvalues \(\lambda_1=\lambda_2\),
so its image in \(\PGL_2(\FF_\ell)\) therefore
has order  \(e\mid \ell\).

We note an interesting and useful fact in passing:
if 
\(\EC/\FF_q\) is vanilla,
\(\ell \not=p\) is an odd prime,
and \(r\) is the number of irreducible factors of \(\Phi_\ell(j(\EC),x)\),
then 
\begin{equation}
    \label{eq:EC-number-of-factors-of-Phi_ell}
    (-1)^r = \left(\frac{q}{\ell}\right)
\end{equation}
(cf.~\cite[Prop.~6.3]{Schoof95}; 
the proof generalizes easily from \(q = p\) to general prime powers).

\subsection{Computing the type of a prime}

The type of a given prime \(\ell\) for \(\EC\)
(that is, being volcanic, Atkin, or Elkies)
is defined in terms of the structure of \(\ZZ[\pi]\) and the trace \(t\).
When we are point-counting, these are unknown quantities;
but we can still determine the type of \(\ell\)
\emph{without} knowing \(t\) or \(\ZZ[\pi]\),
by factoring \(\Phi_\ell(j(\EC),x)\)
and comparing with the possible factorization types above.
This, in turn, gives us useful information about \(t\)
and \(\ZZ[\pi]\).
Determining the type of \(\ell\) in this way
costs $\softO(\ell^2 + (\log q) \ell)$ $\F_q$-operations.

In fact, computing the type of \(\ell\) for \(\EC\)
is a good way of checking the correctness of a claimed modular
polynomial.
Suppose somebody has computed a polynomial \(F(J_1,J_2)\),
and claims it is equal to \(\Phi_\ell\).
The factorization patterns for modular polynomials 
corresponding to the prime types above
are so special that there is very little hope of getting
these patterns for \(F(j(\EC),x)\) for varying \(\EC\) and \(p\)
unless \(F\) and~\(\Phi_\ell\) define the same variety in the
\((J_1,J_2)\)-plane.
We will use the genus-2 analogue of this observation 
in~\S\ref{sec:experiments}
to check the correctness of some of Martindale's modular polynomials.

\subsection{Atkin's improvement}
\label{sec:Atkin}

Atkin's contribution to the SEA algorithm
was to exploit the factorization type of the modular polynomial
to restrict the possible values of \(t\pmod{\ell}\).
While this does not improve the asymptotic complexity of Schoof's algorithm,
it did allow significant practical progress
before the advent of Elkies' improvements.

For example:
if \(\ell\) is volcanic,
then by definition
\begin{equation}
    t^2 = 4q \quad \text{in } \FF_\ell
    \ ,
\end{equation}
which determines \(t_\ell\) up to sign: \(t \equiv \pm2\sqrt{q} \pmod{\ell}\).
Note that this is also a consequence of
Proposition~\ref{prop:EC-trace-relation},
which we will now apply to the other two prime types.

If \(\ell\) is Elkies or Atkin for \(\EC\),
then Proposition~\ref{prop:EC-trace-relation}
tells us that
\begin{equation}
    t^2 = (\zeta + \zeta^{-1}+2)q 
    \quad
    \text{in }
    \FF_\ell
\end{equation}
for some primitive \(e\)-th root of unity \(\zeta\) in \(\FF_{\ell^2}\),
where \(e\mid\ell-1\) if \(\ell\) is Elkies
and \(e\mid \ell+1\) if \(\ell\) is Atkin.
The number of possible values of \(t_\ell^2\)
is therefore half the number of primitive \(e\)-th roots in these cases.
Note that modular polynomials can only give us information about
\(t_\ell^2\)---that is, \(t_\ell\) up to sign---since their solutions tell us about
isogenies only up to quadratic twists,
and twisting changes the sign of the trace.

Obviously, the smaller the degree \(e\) of the non-linear factors of
\(\Phi_\ell(j(\EC),x)\),
the fewer the values that \(t_\ell\) can possibly take.
For example, 
if \(e = 2\) then \(t_\ell = 0\);
if \(e = 3\),
then \(t_\ell = \pm\sqrt{q}\) in \(\FF_\ell\);
and if \(e = 4\),
then \(t_\ell = \pm\sqrt{2q}\) in \(\FF_\ell\).

The challenging part of Atkin's technique
is making use of these extra modular congruences.
Atkin's \emph{match-and-sort} algorithm 
(see eg.~\cite[\S11.2]{Lercier97})
is a sort of sophisticated baby-step giant-step in \(\EC(\FF_q)\)
exploiting this modular information.
Alternatively, we can use 
Joux and Lercier's \emph{Chinese-and-match} algorithm~\cite{JoLe01}.

\subsection{Elkies' improvement}
\label{sec:Elkies}

Elkies' contribution to the SEA algorithm
was to note that when computing \(t_\ell\),
we can replace \(\EC[\ell]\)
with the kernel of a rational \(\ell\)-isogeny, if it exists.
Looking at the classification of primes,
we see that there exists a rational \(\ell\)-isogeny
precisely when \(\ell\) is volcanic or Elkies (whence the terminology).
Of course, as we saw above, 
if \(\ell\) is one of the rare volcanic primes
then \(t_\ell\) is already determined up to sign;
it remains to see what can be done for Elkies primes.

Let \(\ell\) be an Elkies prime for \(\EC\),
and let \(\phi_1\) and \(\phi_2\) be \(\ell\)-isogenies
corresponding to the two roots of \(\Phi_\ell(j(\EC),x)\) in \(\FF_q\).
First, we note that 
\(\pi(P_i) = [\lambda_i]P_i\) for \(P_i\) in \(\ker\phi_i\),
and \(\lambda_1 + \lambda_2 \equiv t \pmod{\ell}\).
We only need to compute one of the \(\lambda_i\),
since then the other
is determined by the relation \(\lambda_1\lambda_2 = q\).

So let \(\phi\) be one of the two \(\ell\)-isogenies;
we want to compute its eigenvalue \(\lambda\).
The nonzero elements \((x,y)\) of \(\ker\phi\)
satisfy \(f_{\phi}(x) = 0\),
where \(f_{\phi}\) is a polynomial of degree \((\ell-1)/2\) (if \(\ell\) is
odd; if \(\ell = 2\), then \(\deg f_{\phi} = 1\)).
To compute \(\lambda\),
we define the ring 
\(A = \FF_q[X,Y]/(f_{\phi}(X),Y^2 - X^3 - aX - b)\),
set \(P = (X,Y)\) in \(\EC(A)\),
then compute \(Q = \pi(P)\)
and solve for \(\lambda\) in \(Q = [\lambda]P\);
then \(t_\ell \equiv \lambda + q/\lambda \pmod{\ell}\).

This approach is substantially faster than Schoof's algorithm for Elkies
\(\ell\),
because the degree of \(f_{\phi}\) is only \((\ell-1)/2\),
whereas the degree of \(\Psi_\ell\) is \((\ell^2-1)/2\);
so each operation in \(\EC(A)\) costs much less than it would if we used
\(\Psi_\ell\) instead of \(f_{\phi}\).
(In practice, it is also nice to be able to reduce the number of costly
Frobenius computations, since we only need to compute \(\pi(P)\) and not
\(\pi(\pi(P))\).)

The crucial step is computing \(f_{\phi}\)
given only \(\EC\) 
and the corresponding root \(j_i\) of \(\Phi_\ell(j(\EC),X)\).
We can do this using Elkies' algorithm,
which is explained in~\cite[\S\S7--8]{Schoof95}.
The total cost of computing \(t_\ell\) is then \(\softO(\log^3q)\)
\(\FF_q\)-operations:
that is, a whole factor of \(\log q\) faster
compared to Schoof's algorithm.

Ideally, then, we should choose \(\mathcal{L}\)
to only contain Elkies and volcanic primes:
that is, non-Atkin primes.
The usual naive heuristic on prime classes is to suppose that 
as \(q \to \infty\), the number
of Atkin and non-Atkin primes less than \(B\) for \(\EC/\FF_q\)
is approximately equal when \(B \sim \log q\);
under this heuristic, 
taking \(\mathcal{L}\) to contain only non-Atkin primes,
the SEA algorithm computes \(t\) in \(\softO(\log^4q)\)
\(\FF_q\)-operations.

While the heuristic holds on the average, assuming the GRH,
Galbraith and Satoh have shown that it can fail for some
curves~\cite[Appendix~A]{Satoh2002}:
there exist curves \(\EC/\FF_q\) such that if we try to compute
\(t_\ell\)
using \(\ell\) in the smallest possible set \(\mathcal{L}\)
containing only non-Atkin primes,
then \(\mathcal{L}\) must contain primes in \(\Omega(\log^2q)\).

\begin{remark}
    It is important to note that Elkies' technique applies only to primes
    \(\ell\) where there exists a rational \(\ell\)-isogeny:
    that is, only Elkies and volcanic primes.
    Atkin's technique for restricting the possible values of \(t_\ell\)
    applies to \emph{all} primes---not only Atkin primes.
\end{remark}

\section{
    The genus 2 setting
}
\label{sec:genus-2-point-counting}

Let \(\C\) be a genus-2 curve
defined over $\F_q$ (again, for $q$ odd).
We suppose that \(\C\) is defined by 
an equation of the form $y^2 = f(x)$, 
where $f$ is squarefree of degree 5.\footnote{%
    For full generality,
    we should also allow \(\deg f = 6\);
    the curve \(\C\) then has two points at infinity.
    This substantially complicates the formul\ae{}
    without significantly modifying the algorithms or their asymptotic complexity,
    so we will not treat this case here.
}
The curve \(\C\) then has a unique point at infinity, which we denote \(\infty\).

\subsection{The Jacobian}

We write \(\JC\) for the Jacobian of \(\C\).
Our main algorithmic handle on \(\JC\) is 
Mumford's model for hyperelliptic Jacobians,
which represents the projective \(\JC\) 
as a disjoint union of three affine subsets.
In this model,
points of \(\JC\) correspond to pairs of polynomials \(\Mumford{a(x)}{b(x)}\)
where \(a\) is monic,
\(\deg b < \deg a \le 2\),
and \(b^2 \equiv f \pmod{a}\)
(we call \(\Mumford{a}{b}\) the \emph{Mumford representation} of the
Jacobian point).
Mumford's coordinates on the affine subsets of \(\JC\)
are the coefficients of the polynomials \(a\) and \(b\)
(and in particular, 
a point \(\Mumford{a}{b}\) of \(\JC\) is defined over \(\FF_q\) 
if and only if \(a\) and \(b\) have coefficients in \(\FF_q\)).
The three affine subsets are
\begin{align*}
    W_2 & := \left\{ \Mumford{a}{b} \in \JC \mid \deg(a) = 2 \right\}
    & \text{(``general'' elements)}
    \ ,
    \\
    W_1 & := \left\{ \Mumford{a}{b} \in \JC \mid \deg(a) = 1 \right\}
    & \text{(``special'' elements)}
    \ ,
    \\
    W_0 & := \left\{ 0_{\JC} = \Mumford{1}{0} \right\}
    & \text{(the trivial element)}
    \ ,
\end{align*}
and \(\JC = W_2\sqcup W_1\sqcup W_0\).
The group law on \(\JC\) can be explicitly computed on Mumford representatives
using Cantor's algorithm~\cite{Cantor87}.

The point of \(\JC\) corresponding to a general divisor class 
\([(x_P,y_P) + (x_Q,y_Q) - 2\infty]\) on \(\C\)
is represented by \(\Mumford{a}{b}\)
where \(a(x) = (x - x_P)(x - x_Q)\)
and \(b\) is the linear polynomial 
such that \(b(x_P) = y_P\) and \(b(x_Q) = y_Q\).
Special classes \([(x_P,y_P) - \infty]\)
are represented by \(\Mumford{a}{b} = \Mumford{x - x_P}{y_P}\),
while \(0_{\JC} = [0]\) is represented by \(\Mumford{a}{b} = \Mumford{1}{0}\).

\subsection{Frobenius and endomorphisms of \texorpdfstring{\(\JC\)}{JC}}
The characteristic polynomial \(\chi_{\pi}\)
of the Frobenius endomorphism \(\pi\) has the form
$$ 
    \chi_{\pi}(X) = X^4 - tX^3 + (2q+s)X^2 -tqX + q^2
    \ ,
$$
where \(s\) and \(t\) are integers satisfying
the inequalities (cf.~\cite{Ruck90})
\begin{align*}
    |s| & < 4q 
    \ ,
    & 
    |t| & \leq 4 \sqrt{q} 
    \ ,
    &
    t^2 & > 4s 
    \ ,
    &
    s + 4q & > 2|t| \sqrt{q}
    \ .
\end{align*}
We have
\[ 
    \#\JC(\F_q)=\chi_{\pi}(1)=1-t+2q+s-tq+q^2
    \ ,
\]
as well as \( \#\cC(\F_q)=1-t+q \) and \( \#\cC(\F_{q^2})=1-t^2+4q+2s+q^2 \).
In genus 2, therefore,
the point counting problem is to determine the integers $s$ and $t$.

\subsection{Real multiplication}
\label{sec:RM}

We are interested in Jacobians \(\JC\) 
with real multiplication by a fixed order \(\mathcal{O}\) 
in a quadratic real field \(F:=\QQ(\sqrt{\Delta})\);
that is, 
such that there is an embedding \(\iota \colon \mathcal{O} \to \End(\JC)\).
In this article,
we will further restrict to the case where \(\mathcal{O}\) 
is the maximal order \(\mathcal{O}_F\) of \(F\); 
note that if $\mathcal{O}$ is an order in $F$ 
that is not locally maximal at a prime $\ell$, 
then there exist no isogenies of degree $\ell$ that preserve the polarization 
(see Definition~\ref{def:mu-isogeny}).
These Jacobians can be constructed either from points in their moduli
spaces (as in~\S\ref{sec:invariants}),
or from a few known explicit families (as in~\S\ref{sec:experiments}).

The fixed field
\(\QQ(\pi + \dualof{\pi})\) of the Rosati involution on \(\QQ(\pi)\)
is a real quadratic field,
and \(\ZZ[\pi + \dualof{\pi}]\) is a suborder of
\(\mathcal{O}_F\).
The characteristic polynomial of \(\pi + \dualof{\pi}\) is
\[
    \chi_{\pi+\dualof{\pi}}(X) = (X^2 - tX + s)^2 \ ,
\]
so determining \(\chi_{\pi+\dualof{\pi}}\)
also solves the point counting problem for \(\JC\).

Later, we will be particularly interested in \(\C\)
such that \(\JC\) has real multiplication 
by an order of small discriminant.
While such curves are special,
from a cryptographic perspective they are not ``too special''.
From an arithmetic point of view,
all curves (with ordinary simple Jacobians) over \(\FF_q\) have real
multiplication.
Here, we simply require that real multiplication to have small
discriminant;
the discriminant of the entire endomorphism ring of \(\JC\) 
can still be just as large as for a general choice of curve over the same field.
From a geometric point of point view,
the moduli of these~\(\C\) live on two-dimensional Humbert surfaces
inside the three-dimensional moduli space of genus-2 curves.
In concrete terms, this means that 
when selecting random curves over a fixed \(\FF_q\),
only \({\sim 1/q}\) of them have real multiplication by a fixed order;
but if we restrict our choice to those curves then there are still
\(O(q^2)\) of them to choose from.

\subsection{From Schoof to Pila}
\label{sec:Pila}

The Schoof--Pila algorithm deals with higher
dimensions \cite{Schoof85,Pila90}.
Its input is a set of defining equations for a projective model of the
abelian variety, and its group law.
Jacobians of genus-2 curves are abelian varieties,
and we can apply Pila's algorithm to them
using the defining equations computed by Flynn~\cite{Flynn90}
or Grant~\cite{Grant90}.
However, the complexity of Pila's algorithm
is \(O((\log q)^\Delta)\),
where \(\Delta\) (and the big-O constant)
depends on the number of variables (i.e., the dimension of the ambient
projective space)
and the degree and number of the defining equations.
Pila derives an upper bound for \(\Delta\) in~\cite[\S4]{Pila90},
but when we evaluate this bound in the parameters of Flynn's model
for \(\JC\)
(72 quadratic forms in 16 variables)
we get a 30-bit \(\Delta\);
Grant's model (13 quadratic and cubic forms in 9 variables)
yields a 23-bit \(\Delta\).\footnote{%
     With polynomial time estimates like these, who needs enemies?
}
While these are only upper bounds,
we are clearly in the realm of the impractical here.

\subsection{The Gaudry--Schost approach}
\label{sec:Gaudry--Schost}
Pila's algorithm requires a concrete (and necessarily complicated)
nonsingular projective model for \(\JC\).
The Gaudry--Schost algorithm applies essentially the same ideas
to Mumford's affine models for subsets of \(\JC\).

Our first problem is to find an analogue for \(\JC\)
of the elliptic division polynomials~\(\Psi_\ell\).
Ultimately,
we want an ideal \(I_\ell = (F_0,\ldots,F_r) \subset \FF_q[A_1,A_0,B_1,B_0]\)
such that \(\Mumford{a}{b} = \Mumford{x^2 + a_1x + a_0}{b_1x + b_0}\)
is in \(\JC[\ell]\) 
if and only if \((a_1,a_0,b_1,b_0)\) is in the variety of \(I_\ell\):
that is, 
\[
    [\ell]\Mumford{x^2+a_1x+a_0}{b_1x+b_0} = 0
    \iff
    F(a_1,a_0,b_1,b_0) = 0 \text{ for all } F \in I_\ell
    \ .
\]
Then, 
the image of \(\Mumford{x^2 + A_1x + A_0}{B_1x + B_0}\)
in \(\JC(\FF_q[A_1,A_0,B_1,B_0]/I_\ell)\) 
is an element of order \(\ell\) 
that we can use for a Schoof-style computation 
of \(\chi(T) \pmod{\ell}\).

The simplest approach here would be to take a general Mumford representative
\(\Mumford{x^2 + A_1x + A_0}{B_1x + B_0}\),
compute \(L = [\ell]\Mumford{x^2 + A_1x + A_0}{B_1x + B_0}\),
and then equate coefficients in \(L = 0_{\JC}\)
to derive the relations in \(I_\ell\).
But we cannot do this, 
because \(L\) is in \(W_2(\FF_q(A_1,A_0,B_1,B_0))\)
(that is, its \(a\)-polynomial has degree 2, and its \(b\)-polynomial
degree 1),
while \(0_{\JC} = \Mumford{1}{0}\) is in \(W_0\):
these elements are not in the same affine subvariety, and cannot be
directly compared or equated in this form.

Gaudry and Harley~\cite{GauHar00} neatly stepped around this problem
by observing that any element of \(\JC\) can be written as the difference 
of two elements of \(W_1\) (which may be defined over a quadratic extension).
They therefore start with 
\(D = [(x_P,y_P)+(x_Q,y_Q)-2\infty] = [(x_P,y_P)-(x_Q,-y_Q)]\) in \(\JC\),
and find polynomial relations on \(x_P\), \(y_P\), \(x_Q\), and \(y_Q\)
such that \([\ell]D = 0\)
by computing \([\ell]\Mumford{x-x_P}{y_P}\)
and \([\ell]\Mumford{x-x_Q}{-y_Q}\),
and equating coefficients in 
\([\ell]\Mumford{x-x_P}{y_P} = [\ell]\Mumford{x-x_Q}{-y_Q}\).
There is a quadratic level of redundancy in these relations,
which is a direct result of the redundancy in the initial
representation of \(D\):
the involution \((x_P,y_P)\leftrightarrow(x_Q,y_Q)\) fixes \(D\).

Gaudry and Schost remove this redundancy by 
resymmetrizing the relations
with respect to this involution,
re-expressing them in terms of
\(A_1 = -(x_P + x_Q)\),
\(A_0 = x_Px_Q\), 
\(B_1 = (y_P-y_Q)/(x_P-x_Q)\),
and
\(B_0 = (x_Py_Q-x_Qy_P)/(x_P-x_Q)\),
and computing a triangular basis
for the resulting \emph{division ideal} \(I_\ell\).
Their algorithm yields a triangular basis for \(I_\ell\),
which facilitates fast reduction modulo \(I_\ell\).

Once we have \(I_\ell\),
we can compute \(t\pmod{\ell}\) and \(s\pmod{\ell}\)
as follows:
\begin{enumerate}
    \item
        Construct the symbolic \(\ell\)-torsion point
        \[
            P 
            := 
            \Mumford{x^2 + A_1x + A_0}{B_1x + B_0}
            \in 
            \JC(\FF_q[A_1,A_0,B_1,B_0]/I_\ell)
            \ ;
        \]
    \item \label{item:compute-Qs-Qt-R-mod-ell}
        Compute the points
        \begin{align*}
            Q_s & := \pi^2(P) \ ,
            \\
            Q_t & := \pi(\pi^2(P) + [q\bmod{\ell}]\pi(P)) \ ,
            \\
            R & := \pi^4(P) + [2q\bmod{\ell}]\pi^2(P) + [q^2\bmod{\ell}]P
        \end{align*}
        using Cantor arithmetic, 
        with reduction of coefficients modulo \(I_\ell\);
    \item \label{item:compute-s-t-mod-ell}
        Search for \(0 \le s_\ell, t_\ell < \ell\)
        such that 
        \[
            [t_\ell]Q_t - [s_\ell]Q_s = R
        \]
        (using, say, a two-dimensional baby-step giant-step algorithm).
\end{enumerate}
The result is an algorithm that runs in time \(\softO(\log^8q)\).
Of course,
once \(t\) has been determined, we can simplify Steps 
\eqref{item:compute-Qs-Qt-R-mod-ell} and
\eqref{item:compute-s-t-mod-ell} above
to find \(s_\ell\) more quickly for the remaining \(\ell\),
but this does not change the asymptotic complexity.
In practice, 
the algorithm has been used 
to construct cryptographically secure curves:
Gaudry and Schost computed a generic genus-2 curve 
over \(\FF_{2^{127} - 1}\)
such that both the Jacobian and its quadratic twist 
have prime order~\cite{GauSch08}.
Instances of the discrete logarithm problem in this Jacobian 
offer a claimed security level of roughly 128 bits,
which is the current minimum for serious cryptosystems.
This computation also represents the current record 
for point counting for general genus-2 curves.

The Gaudry--Schost computation illustrates
not only the state-of-the-art of genus-2 point counting, 
but also the practical challenge involved 
in producing cryptographically strong genus-2 Jacobians.
The Schoof-like point counting algorithm was only applied using
the prime powers $2^{17}$, $3^9$, $5^4$, and $7^2$, 
and the primes \(11\) through \(31\).
Combining the information given by these prime powers 
completely determines $t$, but not \(s\);
but it still gives us enough modular information about
\(s\) to be able to recover its precise value 
using Pollard's kangaroo algorithm 
in a reasonable time ($\approx2$ hours, in this case).
The kangaroo algorithm is exponential, and would not be practical for
computing this Jacobian order alone without the congruence data generated
by the Schoof-like computations.
Gaudry and Schost estimated the average cost of these calculations as one core-month (in 2008) per curve.

\subsection{Point counting with efficiently computable RM}

In~\cite{Gaudry--Kohel--Smith},
Gaudry, Kohel, and Smith
described a number of improvements to the Gaudry--Schost algorithm
that apply when \(\JC\) is equipped with an explicit and efficiently computable
endomorphism \(\phi\) generating a real quadratic subring of \(\End(\JC)\).
When we say that \(\phi\) is \emph{explicit}
we mean that we can compute the images under \(\phi\) 
of divisor classes on \(\JC\),
including symbolic Mumford representatives for generic divisor classes.
When we say that \(\phi\) is \emph{efficiently computable},
we mean that these images can be computed for a cost comparable with a
few group operations: that is, from an algorithmic point of view, 
we may view evaluation of \(\phi\) as an elementary group operation
like adding or doubling.

Suppose that \(\ZZ[\pi + \dualof{\pi}]\) is contained in \(\ZZ[\phi]\)
(this is reasonable, since in the examples we know, \(\ZZ[\phi]\) is
a maximal order), and let \(\Delta\) be the discriminant of \(\ZZ[\phi]\).
Then \(\pi + \dualof{\pi} = m\phi + n\) for some \(m\) and \(n\),
which completely determine \(s\) and \(t\):
if the characteristic polynomial of \(\phi\)
is \(\chi_\phi(X): (X^2 - t_\phi X + s_\phi)^2 \),
then \(t = 2m + nt_\phi\)
and \(s = (t^2 - s_\phi^2\Delta)/4\).
It follows that \(m\) and \(n\) are both in \(O(\sqrt{q})\).

We can compute \(m\) and \(n\)
using a technique similar to Gaudry--Schost.
Multiplying the relation \(\pi + \dualof{\pi} = m\phi + n\) through by \(\pi\),
we have \(\pi^2 - (m\phi + n)\pi + q = 0\).
Imitating Schoof's algorithm,
we can compute \(m_\ell := m \pmod{\ell}\)
and \(n_\ell := n \pmod{\ell}\)
by taking a generic element \(D\) of \(\JC[\ell]\)
(as in Gaudry--Schost),
computing \((\pi^2 + q)(D)\),
\(\pi(D)\), and \(\phi\pi(D)\) (using two applications of \(\pi\)),
and then solving for \(m_\ell\) and \(n_\ell\).

We can do even better by exploiting split primes in \(\ZZ[\phi]\).
If \(\ell = \frakell_1\frakell_2\) is split,
then the \(\ell\)-torsion decomposes
as \(\JC[\frakell_1]\oplus\JC[\frakell_2]\),
and once we have found a short generator (or generators) for
\(\frakell_i\) we can take \(D\) to be an element of \(\JC[\frakell_i]\)
instead of \(\JC[\ell]\).
Such generators can be found with coefficients in \(O(\sqrt{\ell})\);
the result is that we work modulo a much smaller ideal,
of degree \(O(\ell^2)\) rather than \(O(\ell^4)\).

But going further,
\(\pi + \dualof{\pi}\) acts as a scalar on \(\JC[\frakell_i]\),
and so we can compute its eigenvalue to determine \(m_\ell\) and
\(n_\ell\).  The total cost of computing \(m_\ell\) and \(n_\ell\),
and hence \(t_\ell\) and \(s_\ell\),
is then \(\softO(\log^5q)\)~\cite[Theorem 1]{Gaudry--Kohel--Smith},
a substantial improvement on Gaudry--Schost's \(\softO(\log^8q)\).

The computation resembles what we would do for an Elkies
prime in the elliptic case, 
except that there is no need for modular polynomials to compute the prime type, or for an
analogue of Elkies' algorithm: we know in advance which primes split in
\(\ZZ[\phi]\), and we can compute the kernel using the decomposition.
But if we did have an analogue of Elkies' algorithm,
then we could further reduce the complexity by further decomposing
some of the \(\JC[\frakell_i]\) into cyclic factors,
and thus working modulo ideals of degree \(O(\ell)\).
If we have an analogue of Atkin's algorithm,
then we can restrict the possible values of \(m_\ell\) and \(n_\ell\);
this would not change the asymptotic complexity of the algorithm,
but it could have a significant practical impact.

\subsection{Generalizing Elkies' and Atkin's improvements to genus 2}

Ultimately, we would like to generalize the SEA algorithm to genus 2.
The first requirement is a genus-2 analogue of elliptic modular polynomials;
so assume for the moment that we have a modular ideal 
relating suitable invariants of genus-2 curves.

To generalize Elkies' improvements to genus 2,
we need an analogue of Elkies' algorithm:
that is, an algorithm which, given two general moduli points corresponding to
isogenous Jacobians, constructs defining polynomials for (the kernel of)
the isogeny.
The most convenient such presentation would be as an ideal
cutting out the intersection of the kernel with \(W_2\),
since then the Gaudry--Schost approach could be adapted 
without too much difficulty (at least in theory). 
Unfortunately, at present, no such algorithm is known.

In contrast, Atkin's techniques for elliptic curves require only 
the factorization of (specializations of) elliptic modular polynomials;
we deduce possible congruences on the trace from the degrees
of the factors.
It is clear how we should generalize Atkin's techniques to genus 2:
we should deduce possible congruences on \(s\) and \(t\)
from the degrees of primary components of specialized modular ideals.

The following sections make this concrete.
In \S\ref{sec:invariants}, we define the appropriate analogues of the
elliptic \(j\)-invariant for genus-2 curves with real multiplication.
We can then define real-multiplication analogues of the elliptic modular
polynomials in \S\ref{sec:modular-ideals},
before investigating their factorization in \S\ref{sec:Atkin-g2}.

\subsection{\texorpdfstring{\(\mu\)}{mu}-isogenies}
Before defining any generalized invariants or modular polynomials,
we must define an appropriate class of isogenies in genus~2:
that is, isogenies that are compatible with the real
multiplication structure.
(This is not an issue for elliptic curves, 
because the elliptic analogue of the real endomorphism subring 
is just \(\ZZ\)---and everything is compatible with integer
multiplications.)

\begin{definition}
    \label{def:mu-isogeny}
    Let \((\AV,\xi,\iota)\)
    and \((\AV',\xi',\iota')\)
    be triples encoding principally polarized abelian surfaces
    with real multiplication by \(\mathcal{O}_F\).
    Here $\xi\colon\AV \to \AV^{\vee}$ 
    and $\xi'\colon\AV' \to (\AV')^{\vee}$ 
    are principal polarizations,
    and $\iota\colon \mathcal{O}_{F} \hookrightarrow \End(\AV)$ 
    and $\iota'\colon \mathcal{O}_{F} \hookrightarrow \End(\AV')$ 
    are embeddings that are stable under the Rosati involution.
    If \(\mu\) is a totally positive element of $F$, 
    then a \emph{$\mu$-isogeny} 
    $(\AV,\xi,\iota) \to (\AV',\xi',\iota')$ 
    is an isogeny $f \colon \AV \to \AV'$ 
    such that the diagrams
    \[
        \xymatrix{
            \AV
            \ar[r]^{\iota(\mu)}
            \ar[d]_{f}
            &
            \AV\ar[r]^{\xi}
            &
            \AV^\vee
            \\
            \AV'
            \ar[rr]_{\xi'}
            &
            &
            (\AV')^\vee
            \ar[u]_{f^\vee}
        }
        \qquad
        \text{and}
        \qquad
        \xymatrix{
            F
            \ar[r]^{\iota}
            \ar[dr]_{\iota'} 
            & 
            *+[r]{\End(\AV) \otimes \mathbb{Q}} 
            \ar[d]^{\phi}
            \\
            & 
            *+[r]{\End(\AV') \otimes \mathbb{Q}}
        }
    \]
    commute, where $\phi$ is the map induced by $f$ on endomorphism algebras.
\end{definition}
    
If \(f \colon (\AV,\xi,\iota) \to (\AV',\xi',\iota')\) is a \(\mu\)-isogeny,
then the polarization \(\xi'\) pulls back via \(f\) to \(\xi\circ\iota(\mu)\).
For comparison, 
an elliptic \(\ell\)-isogeny is an \(f \colon \EC \to \EC'\)
such that the canonical polarization on \(\EC'\) pulls 
back via \(f\) to \(\ell\) times the polarization on \(\EC\)
(in more concrete terms:
the identity point \(0_{\EC'}\) on \(\EC\)
pulls back via \(f\) to a divisor on~\(\EC\) equivalent to
\(\ell\cdot0_{\EC}\)).

\section{
    Invariants
}
\label{sec:invariants}

Elliptic modular polynomials relate isogenous elliptic curves in terms
of their \(j\)-invariants;
their genus-2 analogues must relate invariants of genus-2 Jacobians.
This section describes and relates the various invariants that we will
need.
Since we are dealing with classical constructions in this section,
we work over a field \(k \subseteq \CC\).
However, the resulting algebraic expressions carry over 
to the case where \(k = \FF_q\)
(at least for large enough \(p\)).
All of the results in this section are well-known,
and are shown here for completeness and easy reference;
we refer the reader to \cite{Lan2}, \cite{LauNaeYan15}, \cite{LauYan11}, 
and \cite{Mar16} for further detail.

\subsection{Invariants for RM abelian surfaces}

Let \(F\) be a real quadratic field
with ring of integers \(\mathcal{O}_F\).
We need RM analogues of the elliptic \(j\)-invariant
and elliptic modular polynomials 
for \(\mu\)-isogenies of abelian surfaces with RM by \(\mathcal{O}_F\).
Our first step is to define appropriate replacements for the \(j\)-invariant 
that classify our triples \((A,\xi,\iota)\) up to isomorphism.
Instead of a single \(j\)-invariant,
we will have a triple \((J_1,J_2,J_3)\) of \emph{RM invariants},
which are functions on the corresponding Hilbert modular surface.

The invariants \((J_1,J_2,J_3)\) are constructed as follows.
For a field $k$, we consider the coarse moduli space
$\mathcal{H}_F(k)$ of triples $(\AV,\xi,\iota)$ 
(where as before,
$\AV/k$ is an abelian variety
with a principal polarization $\xi\colon\AV\to\AV^{\vee}$
and an embedding $\iota\colon\mathcal{O}_F\hookrightarrow \End_{k}(\AV)$ 
stable under the Rosati involution). 
Then $\mathcal{H}_F(k)$ is coarsely represented by the Hilbert modular space
$\SL_2(\mathcal{O}_F\oplus\mathcal{O}_F)\setminus(F\otimes\mathbb{H})$
(see~\cite{vdG}),
where
\(
    F \otimes \mathbb{H} := 
    \{ \tau \in F \otimes \mathbb{C} : \Im(\tau) > 0 \}
\)
and for any fractional ideal $\mathfrak{f}$ of $F$, 
\[
    \SL_2(\mathcal{O}_F\oplus\mathfrak{f})
    :=
    \left\{
        \left(\begin{matrix}a& b\\c& d\end{matrix}\right) \in \SL_2(F)
        : a,d \in \mathcal{O}_F,\  b\in\mathfrak{f},\  c\in\mathfrak{f}^{-1}
    \right\}
\]
acts on \(F\otimes\mathbb{H}\) by
\[
    \left(\begin{matrix}a&b\\c&d\end{matrix}\right)
    \cdot
    \tau
    = 
    \frac{a\tau + b}{c\tau + d}
    \ .
\]

\begin{proposition}
    \label{prop:RM-invariants}
    Let $V$ be the Baily--Borel compactification
    of $SL_2(\mathcal{O}_F)\setminus(F\otimes\mathbb{H})$, 
    and $\CC(V)$ the function field of $V$. There exist rational functions
    \(J_1\), \(J_2\), and \(J_3\) 
    on \(V\)
    such that 
    \[
        \CC(V)=\CC(J_1,J_2,J_3)
        \ .
    \]
\end{proposition} 
\begin{proof}
    The transcendence degree of \(\CC(V)\) over $\CC$ is $2$,
    so there exist $2$ algebraically independent functions
    $J_{1}$, $J_{2}$ in $\CC(V)$. 
    Furthermore, $\CC(V)$ is a finite separable field extension
    of $\CC(J_{1},J_{2})$, so it is generated by at most one further element,~$J_{3}$.
\end{proof}

\begin{definition}
    Fixing a choice of rational functions \(J_1\), \(J_2\), and \(J_3\)
    as in Proposition~\ref{prop:RM-invariants},
    we call \((J_1,J_2,J_3)\) the \emph{RM invariants} for \(F\).
\end{definition}

\subsection{Hilbert modular polynomials for RM abelian surfaces}
\label{sec:modular-ideals}

We are now ready to define modular polynomials
for abelian surfaces with RM structure.
For elliptic curves we have a single \(j\)-invariant,
and we can relate \(\ell\)-isogenous \(j\)-invariants 
using a single bivariate polynomial \(\Phi_\ell(X,Y)\).
For our abelian surfaces, we have a tuple of three invariants 
\((J_1,J_2,J_3)\),
and to relate \(\mu\)-isogenous tuples of invariants
we need a \emph{modular ideal} of polynomials in
\(\QQ[X_1,X_2,X_3,Y_1,Y_2,Y_3]\),
such that when we specialize the first three variables in
the \((J_1,J_2,J_3)\) corresponding to the isomorphism class of some triple \((A,\xi,\iota)\),
the result is an ideal cutting out the moduli points
\((J_1',J_2',J_3')\)
for triples \((A',\xi',\iota')\)
that are \(\mu\)-isogenous to \((A,\xi,\iota)\).

The \emph{Hilbert modular polynomials} below 
represent a particularly convenient basis for this ideal.
We refer the reader to~\cite[Chapter 2]{Mar16}
for theoretical details and proofs,
as well as algorithms for computing the polynomials.
Alternatively, Milio's algorithm can be used to compute Hilbert
modular polynomials 
$\Phi_\ell(X,\mathfrak{J}_1,\mathfrak{J}_2)$ and
$\Psi_\ell(X,\mathfrak{J}_1,\mathfrak{J}_2)$,
in time
$O(d_Td_{\mathfrak{J}_2})\tilde{O}(\ell N) + 4 (\ell+1)
\tilde{O}(d_Td_{\mathfrak{J}_2}N) \subseteq
\tilde{O}(d_Td_{\mathfrak{J}_2}\ell N) $
\cite[Th.~5.4.4]{Milio-thesis},
where $N$ is the precision and $d_T, d_{\mathfrak{J}_2}$ are degrees
involved in the computation, see \cite[\S 5.4]{Milio-thesis}.

\begin{definition}
	The {\em Hilbert modular polynomials} 
    \begin{align*}
        & G_\mu(X_1,X_2,X_3,Y_1)
        \ ,
        \\
        & H_{\mu,2}(X_1,X_2,X_3,Y_1,Y_2)
        = H_{\mu,2}^{(1)}(X_1,X_2,X_3,Y_1)Y_2
          + H_{\mu,2}^{(0)}(X_1,X_2,X_3,Y_1)
        \ ,
        \\
        & H_{\mu,3}(X_1,X_2,X_3,Y_1,Y_3)
        = H_{\mu,3}^{(1)}(X_1,X_2,X_3,Y_1)Y_3
          + H_{\mu,3}^{(0)}(X_1,X_2,X_3,Y_1)
    \end{align*}
    in \(\QQ[X_1,X_2,X_3,Y_1,Y_2,Y_3]\) 
    are defined such that
    for all triples
    \((\AV,\xi,\iota)\) and \((\AV',\xi',\iota')\)
    representing points \(\tau\) and \(\tau'\) in
    a certain Zariski-open subset\footnote{ 
        See~\cite[Chapter 2, Section 2]{Mar16} for details on this subset.
        For point counting over large finite fields,
        it is enough to note that since the subset is Zariski open,
        randomly sampled Jacobians with real multiplication by
        \(\mathcal{O}_F\) have their RM invariants in this subset with
        overwhelming probability.

    }
    of the Baily--Borel compactification of
    \(\SL_2(\mathcal{O}_F\oplus\mathfrak{f})\setminus(F\otimes\mathbb{H})\),
    there exists a $\mu$-isogeny
    \(f \colon (\AV,\xi,\iota)\to(\AV',\xi',\iota')\)
    if and only if
    \begin{align*}
	    G_\mu(J_1(\tau),J_2(\tau),J_3(\tau),J_1(\tau')) = 0 \ ,\\
        H_{\mu,2}(J_1(\tau),J_2(\tau),J_3(\tau),J_1(\tau'),J_2(\tau')) = 0 \ ,\\
        H_{\mu,3}(J_1(\tau),J_2(\tau),J_3(\tau),J_1(\tau'),J_3(\tau')) = 0 \ .
    \end{align*}
\end{definition}

The special form of \(G_\mu\), \(H_{2,\mu}\), and \(H_{3,\mu}\)
are very convenient for computations.
If \((J_1,J_2,J_3)\) is a fixed moduli point,
then each root \(\alpha\) of \(G(J_1,J_2,J_3,x)\)
corresponds to a unique \(\mu\)-isogenous moduli point
\[
    \left(J_1',J_2',J_3'\right)
    =
    \left(
        \alpha,
        -\frac{H_{\mu,2}^{(0)}(J_1,J_2,J_3,\alpha)}{H_{\mu,2}^{(1)}(J_1,J_2,J_3,\alpha)},
        -\frac{H_{\mu,3}^{(0)}(J_1,J_2,J_3,\alpha)}{H_{\mu,3}^{(1)}(J_1,J_2,J_3,\alpha)}
    \right)
    \ .
\]
We observe that the action of Galois on the set of \(\mu\)-isogenies
from an RM abelian variety representing \((J_1,J_2,J_3)\)
is completely described by the action of Galois on the roots of
\(G_\mu(J_1,J_2,J_3,x)\);
in particular, over \(\FF_q\),
rational cycles of \(\mu\)-isogenies under Frobenius
correspond to irreducible factors of \(G_\mu(J_1,J_2,J_3,x)\).
From the point of view of Atkin generalizations,
therefore,
we only really need \(G_\mu\) to replace~\(\Phi_\ell\).

\subsection{Invariants for curves and abelian surfaces}
\label{sec:curve-invariants}

We need to relate the RM invariants \((J_1,J_2,J_3)\)
to the invariants for plain old principally polarized abelian surfaces,
and in particular Jacobians of genus 2 curves
without any special RM structure.
The moduli space $\AV_2$ of principally polarized abelian surfaces
is coarsely represented by the Siegel modular space
\(\Sp_2(\ZZ)\backslash\mathbb{H}_2\),
where
\[
    \mathbb{H}_2
    :=
    \left\{
        \tau = \left(
            \begin{matrix}\tau_1&\tau_2\\ \tau_2&\tau_3\end{matrix}
        \right)
        \in
        \text{Sym}_2(\CC): \Im(\tau) > 0
    \right\}
    \ ,
\]
and the symplectic group 
\[
    \Sp_2(\ZZ)
    =
    \left\{ 
        g \in \GL_4(\ZZ)
        :
        g\left(\begin{matrix}0&I_2\\-I_2&0\end{matrix}\right)g^t
        =\left(\begin{matrix}0&I_2\\-I_2&0\end{matrix}\right)
    \right\}
\]
acts on \(\mathbb{H}_2\)
via
\[
    \left(\begin{matrix}a&b\\c&d\end{matrix}\right)
    \cdot
    \tau
    =
    \frac{
        a\tau + b
    }{
        c\tau + d
    }
    \ .
\]

Every rational function on $\Sp_2(\ZZ)\backslash\mathbb{H}_2$
is a quotient of elements of the graded ring of holomorphic Siegel
modular forms for \(\Sp_2(\ZZ)\).
Igusa proved in~\cite{Igusa62} that 
this ring is generated by 
$\psi_4$, $\psi_6$, $\chi_{10}$, and $\chi_{12}$,
where
\[
    \psi_k(\tau)
    =
    \sum_{
        \left(\begin{smallmatrix}a&b\\c&d\end{smallmatrix}\right)
        \in
        P\setminus\Sp_2(\ZZ)
    }
    \mathrm{det}(c\tau+d)^{-k}
\]
is the normalized Eisenstein series of weight $k$ for even integers $k\geq4$
(here $P$ is the standard Siegel parabolic subgroup of $\Sp_2(\ZZ)$),
and
\begin{align*}
    \chi_{10}
    & =
    -2^{-12}\cdot3^{-5}\cdot5^{-2}\cdot7^{-1}\cdot 43867(\psi_4\psi_6-\psi_{10})
    \ ,
    \\
    \chi_{12}
    & =
    2^{-13}\cdot3^{-7}\cdot5^{-3}\cdot7^{-2}\cdot337^{-1}\cdot 131\cdot 593
    (3^2\cdot7^2\psi_{4}^3 + 2\cdot5^3\psi_{6}^{2} - 691\psi_{12})
\end{align*}
are Siegel modular cusp forms of weight $10$ and $12$ respectively. 

Curves of genus 2 are typically classified up to isomorphism
by their Igusa invariants \((j_1,j_2,j_3)\),
or by their Igusa--Clebsch invariants \((A,B,C,D)\).
Since the map $\C\mapsto \JC$ is an open immersion 
of the (coarse) moduli space of genus-2 curves $\mathcal{M}_2$ into $\AV_2$,
the Igusa invariants $j_i$ can be written as rational functions of 
$\psi_4$, $\psi_6$, $\chi_{10}$ and $\chi_{12}$ as follows~\cite{Igusa67}:
\begin{align*}
    j_1(\tau) 
    & =
    2\cdot 3^5\cdot\chi_{12}^5\chi_{10}^{-6}
    \ ,
    \\
    j_2(\tau) 
    & =
    2^{-3}\cdot3^3\cdot\psi_4\chi_{12}^3\chi_{10}^{-4}
    \ ,
    \\
    j_3(\tau)
    & =
    2^{-5}\cdot 3\cdot \left(
        \psi_6\chi_{12}^2\chi_{10}^{-3}
        +
        2^2\cdot3\cdot\psi_4\chi_{12}^3\chi_{10}^{-4}
    \right)
    \ .
\end{align*}
Here $j_i(\tau)=j_i(\C)$ if there is a genus 2 curve $\C/\CC$ 
such that $\JC$ is isomorphic to the abelian surface $\CC^2/(\ZZ^2\tau+\ZZ^2)$.
If there is no such \(\C\),
which happens exactly when $\chi_{10}(\tau)=0$,
then $j_i(\tau)$ is not well-defined.
The Igusa--Clebsch invariants
are related to the Siegel modular forms by
\begin{align}
    \label{eq:psi-from-IgusaClebsch}
    \left( 
        \psi_{4} ,\ 
        \psi_{6} ,\ 
        \chi_{10} ,\ 
        \chi_{12} 
    \right)
    &= 
    \left(
        2^{-2}B ,\ 
        2^{-3}(AB-3C) ,\ 
        -2^{-14}C ,\ 
        2^{-17}3^{-1}AD 
    \right) 
    \ .
\end{align}

\subsection{Pulling back curve invariants to RM invariants}
\label{sec:pullback}
The natural maps
\( \mathbb{H}^2 \to \mathbb{H}_2 \),
\( \SL_2(F) \to \Sp_2(\QQ) \),
and
\( (\mathcal{O}_F/2\mathcal{O}_F)^2 \to (\ZZ/2\ZZ)^4 \)
induce an embedding 
\[
    \phi\colon \mathcal{H}_F(k) \hookrightarrow \AV_2(k)
    \ ,
\]
which we can use to pull back Igusa invariants to RM invariants,
thus expressing the~\(j_i\) in terms of the~\(J_i\).
We will see detailed formul\ae{} 
for this pullback for \(F = \QQ(\sqrt{5})\) 
in Proposition~\ref{pullbacks}.

This pullback
from curves and their invariants to RM invariants is essential for
our computations:
after all, in point counting one usually starts from a curve.
In our applications, we are given the equation of a curve \(\C/\FF_q\)
drawn from a family of curves with known RM by \(\mathcal{O}_F\).
Having computed the Igusa or Igusa--Clebsch invariants of \(\C\),
we can pull them back to RM invariants \((J_1,J_2,J_3)\).
This pullback is possible, because \(\C\) was chosen from an appropriate
family, but choosing a preimage \((J_1,J_2,J_3)\)
implicitly involves choosing one of the two embeddings
of \(\mathcal{O}_F\) into \(\End(\JC)\).
This choice cannot always be made over the ground field:
a point in \(\mathcal{A}_2(k)\)
may not pull back to a pair of points in \(\mathcal{H}_F(k)\),
but rather a conjugate pair of points over a quadratic extension of \(k\).
Proposition~\ref{invformulae} 
makes this subtlety explicit in the case \(F = \QQ(\sqrt{5})\).

\section{
    Atkin theorems in genus 2
}
\label{sec:Atkin-g2}

We are now ready to state some Atkin-style results
for \(\mu\)-isogenies in genus 2.

Let \((\AV,\xi,\iota)\) be a triple
describing a vanilla abelian surface over~\(\FF_q\)
with real multiplication by $\mathcal{O}_F$,
and let \(\mu\) be a totally positive element of \(\mathcal{O}_F\)
of norm \(\ell\).
Then \(\iota(\mu)\)
is an endomorphism of degree \(\ell^2\),
and we have a subgroup\footnote{%
    We emphasize that the subgroup \(\AV[\mu]\) depends on \(\iota\),
    but we have chosen to write \(\AV[\mu]\) instead of the more
    cumbersome \(\AV[\iota(\mu)]\).
}
\[
    \AV[\mu] := \ker(\iota(\mu)) \subset \AV[\ell]
    \ .
\]
If \((\bar\mu) \not= (\mu)\)
(that is, \((\ell)\not=(\mu^2)\)),
then we have a decomposition 
\(
    \AV[\ell] = \AV[\mu] \oplus \AV[\bar\mu]
\).
The one-dimensional subspaces of \(\AV[\mu]\)
are the kernels of \(\mu\)-isogenies.

In~\S\ref{sec:SEA}
we used the elliptic modular polynomial \(\Phi_\ell\)
to study the structure of~\(\EC[\ell]\).
Here, we will use the Hilbert modular polynomial \(G_\mu\) to study the structure of~\(\AV[\mu]\).
The propositions of this section are generalizations for curves of genus $2$
to Schoof's Propositions 6.1, 6.2 and 6.3 for elliptic curves in \cite{Schoof95}.

\subsection{Roots of \texorpdfstring{\(G_\mu\)}{Gmu} and the order of Frobenius}

Our first result relates the order of Frobenius acting on
\(\PP(\AV[\mu])\) to the extensions of $\mathbb{F}_q$ generated by roots of
specialized Hilbert modular polynomials.

\begin{proposition}\label{prop:61}
    Let $\AV/\FF_q$ be a vanilla abelian surface 
    with RM by \(\mathcal{O}_F\)
    and RM invariants \((J_1,J_2,J_3)\) in \(\FF_q^3\),
    and with Frobenius endomorphism \(\pi\).
    Let \(\mu\) be a totally positive element of \(\mathcal{O}_F\)
    of prime norm $\ell=\mu\overline{\mu}$.
    \begin{enumerate}
        \item \label{prop:item:zero-Gmu}
            The polynomial $G_\mu(J_1,J_2,J_3,x)$ 
            has a zero \(\tilde{J}_1\) in \(\mathbb{F}_{q^e}\)
            if and only if the kernel 
            of the corresponding \(\mu\)-isogeny \(\AV\rightarrow\tilde{\AV}\) 
            is a $1$-dimensional eigenspace of $\pi^e$ in $\AV[\mu]$. 
        \item \label{prop:item:split-Gmu}
            The polynomial $G_\mu(J_1,J_2,J_3,x)$ splits completely in
            $\mathbb{F}_{q^e}[x]$ if and only if $\pi^e$ acts as a
            scalar matrix on $\AV[\mu]$.
    \end{enumerate}
\end{proposition}
\begin{proof} 
    The proof follows that of~\cite[Proposition~6.1]{Schoof95}
    (stated as Lemma~\ref{lemma:Schoof-6-1-1} here).
        
    For \eqref{prop:item:zero-Gmu}: 
    Let \(f \colon \AV \to \tilde{\AV}\) be a \(\mu\)-isogeny with kernel \(S\),
    and let \((\tilde{J}_1,\tilde{J}_2,\tilde{J}_3)\)
    be the RM invariants of \(\tilde{\AV}\).
    If \(S\) is an eigenspace of \(\pi^e\),
    then the quotient \(\AV \to \AV/S\) is defined over \(\FF_{q^e}\).
    The Igusa invariants of \(\AV/S\) are therefore all in \(\FF_{q^e}\),
    and since \(\AV/S\) is isomorphic to \(\tilde{\AV}\)
    as a principally polarized abelian surface,
    the Igusa invariants of \(\tilde{\AV}\) are all in \(\FF_{q^e}\).
    To conclude that \(\tilde{J}_1\) is in \(\FF_{q^e}\),
    we need to show that the injection
    \(\tilde{\iota}\colon\mathcal{O}_F\hookrightarrow\End(\tilde{\AV})\) 
    is defined over \(\FF_{q^e}\);
    but this follows from the commutativity of the second diagram in
    Definition~\ref{def:mu-isogeny}.

    Conversely: suppose 
    $G_{\mu}(J_1,J_2,J_3,\tilde{J}_1)=0$
    for some \(\tilde{J}_1\) in \(\FF_{q^e}\).
    Then the fact that each of the \(H_{\mu,i}\)
    is a linear polynomial in \(Y_i\)
    with coefficients in \(\FF_q[J_1,J_2,J_3,\tilde{J}_1] = \FF_{q^e}\)
    shows that there exist \(\tilde{J}_2\) and \(\tilde{J}_3\)
    in \(\FF_{q^e}\)
    such that \((\tilde{J}_1,\tilde{J}_2,\tilde{J}_3)\)
    are the RM invariants of a triple \((\tilde{\AV},\tilde{\xi},\tilde{\iota})\)
    that is \(\mu\)-isogenous to \((\AV,\xi,\iota)\).
    This means that there is an \(\FFbar_q\)-isomorphism
    \(
        (\tilde{\AV},\tilde{\xi},\tilde{\iota})
        \to 
        (\AV',\xi',\iota') 
    \)
    where \((\AV',\xi',\iota')\) 
    is defined over~\(\FF_{q^e}\).
    Let \(f \colon \AV \to \AV'\)
    be the composite \(\mu\)-isogeny.
    Its kernel \(S\)
    is a one-dimensional subspace of \(\AV[\ell]\).
    It remains to show that \(S\)
    is an eigenspace of \(\pi^e\);
    this is the case if and only if \(f\) is defined over \(\FF_{q^e}\).
    The \(\ZZ\)-module 
    $\Hom_{\FFbar_q}(\AV,\AV')$
    is free of rank 4 (because \(\AV\) is vanilla);
    and its submodule
    $\Hom_{\FF_{q^e}}(\AV,\AV')$ of \(\FF_{q^e}\)-isogenies 
    is either 0 or equal to \(\Hom_{\FFbar_q}(\AV,\AV')\).
    Hence,
    \(f\) is defined over \(\FF_{q^e}\)
    if \(\Hom_{\FF_{q^e}}(\AV,\AV')\not=0\);
    and \(\Hom_{\FF_{q^e}}(\AV,\AV')\not=0\)
    if and only if the Frobenius endomorphisms 
    of \(\AV/\FF_{q^e}\) and \(\AV'\) have the same characteristic
    polynomial.

    Since \(\AV\) is vanilla,
    and \(\AV'\) is \(\FFbar_q\)-isogenous to \(\AV\),
    we have
    \( \End_{\FFbar_q}(\AV')\otimes\QQ
        \cong
        \End_{\FFbar_q}(\AV)\otimes\QQ
        \cong
        K
    \)
    for some quartic CM-field \(K\).
    So let \(\psi\) and \(\psi'\)
    be the images in \(K\) of the Frobenius endomorphisms 
    of \(\AV/\FF_{q^e}\) and \(\AV'\), respectively
    (note that \(\psi = \pi^e\)).
    Now up to complex conjugation,
    we have \(\psi^s = (\psi')^s\) in \(K\)
    for some \(s > 0\).
    If $\psi=\psi'$, then 
    $\AV$ and $\AV'$ are $\FF_{q^e}$-isogenous,
    and we are done.
    If $\psi=-\psi'$, 
    then we replace $(\AV',\xi',\iota')$
    by its quadratic twist;
    and then $\AV$ and $\AV'$ are $\FF_{q^e}$-isogenous.
    Otherwise, if $\psi\neq\pm\psi'$,
    then $\psi/\psi'$ must be a root of unity of order at least $3$ in \(K\),
    which is impossible because \(\AV\) is vanilla.
    Hence \(\psi = \psi'\),
    so \(\psi\) and \(\psi'\) have the same characteristic polynomial,
    and therefore \(f\) is defined over \(\FF_{q^e}\).
        
    For \eqref{prop:item:split-Gmu}: 
    If all of the zeroes of $G_{\mu}(J_1,J_2,J_3,x)$
    are contained in $\FF_{q^e}$, then 
    all of the $1$-dimensional subspaces of $\AV[\mu]$ 
    are eigenspaces of $\pi^e$ 
    by Part \eqref{prop:item:zero-Gmu}. 
    This implies that $\pi^e$ acts as a scalar matrix on
    \(\AV[\mu]\).
\end{proof}

\begin{remark}
    As an example of what can go wrong if the vanilla condition is dropped, 
    consider the curve
    \[
        \C : y^{2} = x^{5} + 1
        \ .
    \] 
    The Jacobian \(\JC\) of this curve has complex
    multiplication by $\QQ(\zeta_{5})$, so it is not vanilla.
    While \(\JC\) has real multiplication by the maximal order
    of \({\QQ(\sqrt{5})}\),
    the Siegel modular form $\psi_{4}$ is zero for this curve. 
    Proposition~\ref{invformulae} below
    gives explicit formul\ae{} for \(J_1\), \(J_2\), and \(J_3^2\)
    for Jacobians with maximal real multiplication by \(\QQ(\sqrt{5})\);
    and when we look at those formul\ae{},
    we see that \(J_1\) is not well-defined when \(\psi_4 = 0\).
\end{remark}

\subsection{The factorization of \texorpdfstring{\(G_\mu\)}{Gmu}}

The Frobenius endomorphism \(\pi\) of \(\AV\)
commutes with \(\iota(\mu)\) (since \(\AV\) is vanilla),
so it restricts to an endomorphism of \(\AV[\mu]\).

\begin{lemma}
    \label{lemma:technical-1}
    Let \(\AV/\FF_q\) be a vanilla abelian surface
    with Frobenius endomorphism~\(\pi\),
    and let \(\ell\) be an odd prime.
    \begin{enumerate}
        \item \label{lemma:item:ell-splits}
            If \(\ell\) splits in \(\ZZ[\pi + \dualof{\pi}]\)
            (or equivalently, if \(t^2 - 4s\) is a square in \(\FF_\ell\)),
            then
            \(
                \chi_{\pi}(T) 
                \equiv 
                (T^2 - uT + q)(T^2 - u'T + q)
                \pmod{\ell}
            \)
            for some \(u\) and \(u'\) in \(\ZZ/\ell\ZZ\).
        \item \label{lemma:item:ell-ramified}
            If \(\ell\) is ramified in \(\ZZ[\pi + \dualof{\pi}]\)
            (or equivalently, if \(\ell\) divides \(t^2 - 4s\)),
            then 
            \(
                \chi_{\pi}(T) 
                \equiv 
                (T^2 - uT + q)^2
                \pmod{\ell}
            \)
            where \(u = t/2\) in \(\ZZ/\ell\ZZ\).
    	\item  \label{lemma:item:ell-inert}
            If $\ell$ is inert in $\ZZ[\pi+\dualof{\pi}]$ 
            (or equivalently, if $t^2 - 4s$ is a square in $\FF_\ell$), 
            then 
            \(
                \chi_{\pi}(T)
                \not\equiv
                (T^2 - uT + q)(T^2 - u'T + q)\pmod{\ell}
            \) 
            for any $u,u'\in\ZZ/\ell\ZZ$.
    \end{enumerate}
\end{lemma}
\begin{proof}
    This is a direct consequence of \cite[Chap.\ 1: Prop.\ 25]{Lan}.
\end{proof}

\begin{lemma}
    \label{lemma:technical-2}
    Let \((\AV,\xi,\iota)\) be a triple
    describing a vanilla abelian surface over~\(\FF_q\)
    with real multiplication by $\mathcal{O}_F$,
    and let \(\mu\) be a totally positive element of \(\mathcal{O}_F\)
    of prime norm \(\mu\overline{\mu}=\ell\).
    The restriction of the Frobenius endomorphism \(\pi\)
    to \(\AV[\mu]\)
    has characteristic polynomial
    \[
        \chi_{\pi,\mu}(T)
        \equiv T^2 - uT + q
        \pmod{\ell}
        \quad
        \text{ for some } u \in \ZZ/\ell\ZZ
        \ .
    \]
\end{lemma}
\begin{proof}
    By definition, \(\ell = \mu\bar\mu\) splits in \(\mathcal{O}_F\),
    so it either splits or ramifies 
    in the suborder \(\ZZ[\pi+\dualof{\pi}] \subseteq \mathcal{O}_F\);
    we are therefore in Case \eqref{lemma:item:ell-splits} or
    \eqref{lemma:item:ell-ramified} 
    of Lemma~\ref{lemma:technical-1}.
    In particular,
    both \(\pi\) and \(\dualof{\pi}\)
    restrict to endomorphisms of \(\AV[\mu]\),
    and they have the same eigenvalues \(\lambda\) and \(q/\lambda\);
    so the characteristic polynomial of \(\pi\)
    is \(T^2 - (\lambda + q/\lambda)T + q\).
    The result follows with \(u = \lambda + q/\lambda\).
\end{proof}

Proposition~\ref{prop:factorization}
uses the factorization of the modular polynomial \(G_\mu\),
specialized at the RM invariants of \(\AV\),
to derive information \(\chi_{\pi,\mu}(T)\pmod{\ell}\).

\begin{proposition}  \label{prop:factorization}
    Let \((\AV,\xi,\iota)\) be a triple
    describing a vanilla abelian surface over~\(\FF_q\)
    with real multiplication by $\mathcal{O}_F$ 
    and with RM invariants \((J_1,J_2,J_3)\),
    and let \(\mu\) be a totally positive element of \(\mathcal{O}_F\)
    of prime norm \(\mu\overline{\mu}=\ell\).
    Let \(\pi\) be the Frobenius endomorphism of \(\AV\),
    with \(\chi_{\pi,\mu}(T) = T^2 - uT + q\) the characteristic polynomial 
    of the restriction of \(\pi\) to \(\AV[\mu]\),
    and let \(e\) be the order of \(\pi\) 
    in \(\Aut(\PP(\AV[\mu])) \cong \PGL_2(\FF_\ell)\). 

    The polynomial
    \(
        G_\mu(J_1,J_2,J_3,x)
    \) 
    has degree \(\ell+1\) in \(\FF_q[x]\),
    and its factorization type is as follows:
    \begin{enumerate}
        \item \label{prop:item:non-square} 
            If \(u^2 - 4q\) is not a square in \(\FF_\ell\),
            then \(e > 1\) and 
            the factorization type is
            \[
                (e,\ldots,e)
                \quad
                \text{where}
                \quad
                e \mid \ell+1
                \ .
            \]
        \item \label{prop:item:square} 
            If \(u^2 - 4q\) is a nonzero square in \(\FF_\ell\),
            then
            the factorization type is 
            \[
            	(1,1,e,\ldots,e)
                \quad
                \text{where}
                \quad
                e\mid \ell-1
                \ .
            \]
        \item \label{prop:item:zero} 
            If \(u^2 - 4q = 0\) in \(\FF_\ell\),
            then the factorization type is
            \[
            	(1,e)
                \quad
                \text{where}
                \quad
                e = \ell
                \ .
            \]
    \end{enumerate}
\end{proposition}
\begin{proof}
    By Lemma~\ref{lemma:technical-1},
    the endomorphism $\pi$ acts on $\AV[\mu]$ 
    as a $2\times2$ matrix in \(\GL_2(\FF_\ell)\)
    with characteristic polynomial \(T^2-uT+q=0\).
    If the matrix has two conjugate eigenvalues $\lambda_1$, $\lambda_2$
    in \(\FF_{\ell^2}\),
    then we are in Case \eqref{prop:item:non-square}: 
    there are no 1-dimensional eigenspaces 
    of \(\pi\) in \(\AV[\mu]\),
    and all irreducible factors of $G_{\mu}(J_1,J_2,J_3,x)$ have degree~$e$,
    where $e$ is the smallest exponent such that $\lambda_{i}^{e}$ 
    is in $\mathbb{F}_{\ell}$.
    
    If the matrix has two eigenvalues in $\mathbb{F}_{\ell}$ 
    and is diagonalizable, then the discriminant $t^2- 4s$ 
    is a square modulo $\ell$:
    we are in Case \eqref{prop:item:square}. 
    This time $\AV[\mu]$ 
    is the direct product of two 1-dimensional eigenspaces,
    which account for
    two linear factors of $G_{\mu}(J_1,J_2,J_3,x)$. 
    The remaining factors have degree~$e$, 
    where $e$ is the smallest positive integer such that 
    $\pi^e$ acts as a scalar matrix. 
    
    If the matrix has a double eigenvalue and is not diagonalizable,
    then we are in Case \eqref{prop:item:zero}: 
    there is only one 1-dimensional eigenspace, and the matrix
    of $\pi^\ell$ is scalar. 
\end{proof}

\subsection{The characteristic polynomial of Frobenius}

Now that we can compute the order of Frobenius,
we want to use this to derive information on the characteristic
polynomial.
Proposition~\ref{prop:AVmu-trace-relation}
generalizes Proposition~\ref{prop:EC-trace-relation}
to genus 2.

\begin{proposition}
    \label{prop:AVmu-trace-relation}
    Let \((\AV/\FF_q,\xi,\iota)\)
    be a triple describing a vanilla abelian surface
    with real multiplication by \(\mathcal{O}_F\),
    and let \(\mu\) be a totally positive element
    of prime norm \(\ell = \mu\bar\mu \not\in \{2,p\}\).
    Let \(\pi\) be the Frobenius endomorphism of \(\AV\),
    and \(\chi_{\pi,\mu}(T) = T^2 - uT + q\)
    the characteristic polynomial of its restriction to \(\AV[\mu]\).
    If \(e\) is the order of the image of \(\pi\) 
    in \(\Aut(\PP(\AV[\mu])) \cong \PGL_2(\FF_\ell)\),
    then 
    \[
        u^2 = \eta_{e}q 
        \quad 
        \text{in } 
        \FF_\ell 
        \ ,
    \]
    where 
    \(
        \eta_{e} 
        = 
        \begin{cases}
            \zeta + \zeta^{-1} + 2
            \text{ with } \zeta \in \FF_{\ell^2}^\times \text{ of order } e
            & 
            \text{if } \gcd(\ell,e) = 1
            \ ,
            \\
            4 & \text{otherwise} \ .
        \end{cases}
    \)
\end{proposition}
\begin{proof}
    The proof is identical to that of
    Proposition~\ref{prop:EC-trace-relation}.
\end{proof}

Coming back to point counting:
suppose we have a Jacobian \(\JC\)
with real multiplication by \(\mathcal{O}_F\);
we want to compute the characteristic polynomial
\[
    \chi_{\pi}(T) = T^4 - tT^3 + (2q + s)T^2 - tqT + q^2
    \ .
\]
If we have a totally positive element \(\mu\) in \(\mathcal{O}_F\)
such that \(\mu\bar\mu = \ell\),
then we know that 
\(\chi_{\pi}(T)\pmod{\ell}\) splits into two quadratic factors:
\[
    \chi_\pi(T) 
    \equiv
    {\chi_{\pi,\mu}(T)}
    {\chi_{\pi,\bar\mu}(T)}
    \equiv
    (T^2 - uT + q)
    (T^2 - u'T + q)
    \pmod{\ell}
    \ ,
\]
so
\begin{equation}
    \label{eq:t-s-u}
    t \equiv u + u' \pmod{\ell}
    \qquad
    \text{and}
    \qquad
    s \equiv uu' - 2q
    \pmod{\ell}
    \ .
\end{equation}

Given precomputed Hilbert modular polynomials \(G_\mu\) and \(G_{\bar\mu}\),
then,
we can specialize them at the RM invariants of \(\JC\)
and factor to determine the order of Frobenius on \(\JC[\mu]\) and
on \(\JC[\bar\mu]\) using Proposition~\ref{prop:factorization}.
We can then apply Proposition~\ref{prop:AVmu-trace-relation} and
Equations~\eqref{eq:t-s-u} to restrict the possible
values of \(s\) and \(t\) modulo \(\ell\).

The question of how best to exploit this extra modular information
remains open.  
Atkin's match-and-sort and 
Joux and Lercier's Chinese-and-match algorithms
for elliptic curves cannot be re-used
directly here, because they were designed to solve the
one-dimensional problem of determining the elliptic trace,
while here we have the two-dimensional problem of determining \((s,t)\).

\subsection{Prime types for real multiplication by \texorpdfstring{\(\mathcal{O}_F\)}{OF}}

The factorization patterns in Proposition~\ref{prop:factorization}
are the same as those we saw for specialized elliptic modular polynomials
in~\S\ref{sec:prime-types}.
This leads us to define an analogous classification of prime types,
for totally positive elements in \(\mathcal{O}_F\) of prime norm.

\begin{definition}
    \label{def:RM-types}
    Let \(\mu\) be a totally positive element of \(\mathcal{O}_F\)
    such that \(\mu\bar\mu = (\ell)\) for some prime \(\ell \not= 2, p\).
    We say that
    \begin{itemize}
        \item
            \(\mu\) is {\bf \(\mathcal{O}_F\)-Elkies} 
            for a vanilla triple \((\AV,\xi,\iota)\) 
            with RM invariants \((J_1,J_2,J_3)\)
            if the factorization type of \(G_\mu(J_1,J_2,J_3,x)\)
            is \((1,1,e,\ldots,e)\) with \(e > 1\);
        \item
            \(\mu\) is {\bf \(\mathcal{O}_F\)-Atkin} 
            for a vanilla triple \((\AV,\xi,\iota)\) 
            with RM invariants \((J_1,J_2,J_3)\)
            if the factorization type of \(G_\mu(J_1,J_2,J_3,x)\)
            is \((e,\ldots,e)\) with \(e > 1\);
            and
        \item
            \(\mu\) is {\bf \(\mathcal{O}_F\)-volcanic} 
            for a vanilla triple \((\AV,\xi,\iota)\) 
            with RM invariants \((J_1,J_2,J_3)\)
            if the factorization type of \(G_\mu(J_1,J_2,J_3,x)\)
            is \((1,e)\) or \((1,\ldots,1)\).
    \end{itemize}
\end{definition}

If \(K \cong \End_{\FFbar_q}(\AV)\otimes\QQ\) is Galois then 
the type of \(\mu\) completely determines the type of \(\bar\mu\) 
(and vice versa). 
For general \(K\), however, this does not hold:
the type of \(\bar\mu\) is not determined by the type of \(\mu\).

\subsection{The parity of the number of factors of \texorpdfstring{\(G_\mu\)}{Gmu}}

The following proposition is 
the genus-2 real multiplication analogue 
of Equation~\eqref{eq:EC-number-of-factors-of-Phi_ell} 
(cf.~\cite[Prop.~6.3]{Schoof95}).

\begin{proposition}
    Let \((\AV,\xi,\iota)\) be a triple
    describing a vanilla abelian surface over~\(\FF_q\)
    with real multiplication by \(\mathcal{O}_F\),
    and with RM invariants \((J_1,J_2,J_3)\).
    Let \(\mu\) be a totally positive element of \(\mathcal{O}_F\)
    of prime norm \(\mu\overline{\mu}=\ell\),
    let \(\chi_{\pi,\mu}(T) = T^2 - uT + q\)
    be the characteristic polynomial of \(\pi\)
    restricted to \(\AV[\mu]\),
    and let $r$ denote the number of irreducible factors 
    in the factorization of \(G_\mu(J_1,J_2,J_3,x)\).
    Then 
    \[
        (-1)^r = \left(\frac{q}{\ell}\right)
        \ .
    \] 
\end{proposition}
\begin{proof}
    If $\ell$ divides $u^2-4q$ 
    and $\pi$ has order $\ell$ 
    in Case \eqref{prop:item:zero} 
    of Proposition~\ref{prop:factorization},
    then the result is true.
    Suppose therefore that $u^2-4q\neq 0 \pmod{\ell}$, 
    that is, 
    we are in Cases \eqref{prop:item:non-square} or
    \eqref{prop:item:square} 
    of Proposition~\ref{prop:factorization},
    and let $\mathcal{T}\subseteq\GL_2(\mathbb{F}_\ell)$ 
    be a maximal torus containing~$\pi$.
    In other words, we take
    \(
        \mathcal{T}
        =
        \{
            \operatorname{diag}(\alpha,\beta):
            \alpha,\beta \in \FF_\ell^\times
        \}
    \)
    split in Case~\eqref{prop:item:square}, 
    and $\mathcal{T}$ non-split
    (i.e., isomorphic to $\FF_{\ell^2}^\times$) 
    in Case~\eqref{prop:item:non-square}. 
    The image $\overline{\mathcal{T}}$ 
    of $\mathcal{T}$ in $\PGL_2(\FF_\ell)$ 
    is cyclic of order $\ell+1$ in Case~\eqref{prop:item:non-square} 
    and $\ell-1$ in Case~\eqref{prop:item:square}. 
    The determinant induces an isomorphism 
    \(
        \det
        \colon
        \overline{\mathcal{T}}/\overline{\mathcal{T}}^2
        \to
        \FF_{\ell}^\times/(\FF_\ell^\times)^2
    \).
    The action of $\pi$ is via $\det(\pi)=q$,
    and we obtain an isomorphism
    \(
        \det
        \colon
        \overline{\mathcal{T}}
        /
        \subgrp{\overline{\mathcal{T}}^2,\pi}
        \to
        \FF_{\ell}^\times/\subgrp{(\FF_{\ell}^\times)^2,q}
    \).
    This shows that the index
    $[\overline{\mathcal{T}}:\pi]$ is odd 
    if and only if 
    $q$ is not a square mod~$\ell$.
    Since the number $r$ of
    irreducible factors of $G_{\mu}(J_1,J_2,J_3,x)$ over $\mathbb{F}_q$ 
    is equal to $r=(l+1)/e$ 
    or $r=2+(l-1)/e=[\overline{\mathcal{T}}:\pi]$,
    the proposition follows.
\end{proof}

\section{
    The case \texorpdfstring{\(F = \QQ(\sqrt{5})\)}{F = Q(sqrt{5})}: Gundlach--M\"{u}ller invariants
}
\label{sec:sqrt5}

All of the theory above can be made much more explicit 
in the case where \(F = \QQ(\sqrt{5})\),
where the invariants \(J_1\), \(J_2\), and \(J_3\)
are known as Gundlach--M\"{u}ller invariants~\cite{Gundlach1963,Muller85}.
Our computational results are based on this case,
so we will work out the details here, 
following the treatment in~\cite{LauYan11}.

Fixing a square root of \(5\) in \(\CC\),
we set
\(\epsilon := (1+\sqrt{5})/2\)
and 
\(\bar\epsilon := (1-\sqrt{5})/2\);
each is the image of the fundamental unit of
\(\mathcal{O}_{\QQ(\sqrt{5})}\) 
under one of its two embeddings into~\(\CC\).
Let
\[
    q_1 
    := 
    e\left(
        \frac{\epsilon z_1 - \overline{\epsilon} z_2}{\sqrt{5}}
    \right)
    \quad
    \text{and}
    \quad 
    q_2
    :=
    e\left(\frac{z_2-z_1}{\sqrt{5}}\right)
    \quad
    \text{for}
    \quad
    z=(z_1,z_2)\in\mathbb{H}^2
    \ .
\]
The Eisenstein series of even weight $k\geq2$ are defined by
\[
    g_k(z)
    =
    1
    +
    \sum_{t=a+b\bar\epsilon \in \mathcal{O}^{+}_{F}}
    b_k(t)q_{1}^aq_{2}^b
    \ ,
\]
where
the coefficients \(b_k(t)\) are defined by
\[
    b_k(t)
    =
    \kappa_k\sum_{(\mu)\supseteq(t)}\text{N}(\mu)^{k-1}
    \quad
    \text{with}
    \quad
    \kappa_k
    =
    \frac{(2\pi)^{2k}\sqrt{5}}{(k-1)!^25^k\zeta_F(k)}
    \in
    \QQ
\]
(here \(\mathrm{N}(\mu)\) is the norm \(\#\mathcal{O}_F/(\mu)\)).
The Hilbert modular forms
\(s_6\), \(s_{10}\), \(s_{12}\), and \(s_{15}\)
of respective weight \(6\), \(10\), \(12\), and \(15\)
for $\mathcal{H}_{\mathbb{Q}(\sqrt{5})}$ 
are defined by
\begin{align*}
    s_6 
    & :=
    -\frac{67}{2^5\cdot3^3\cdot5^2}(g_6-g_{2}^3)
    \ ,
    \\
    s_{10}
    & :=
    \frac{1}{2^{10}\cdot3^{5}\cdot5^{5}\cdot7}
    \left(
        191\cdot2161g_{10}-5\cdot67\cdot2293g_{2}^2g_6+2^2\cdot3\cdot7\cdot4231g_{2}^5
    \right)
    \ ,
    \\
    s_{12}
    & :=
    \frac{1}{2^2}\left( s_{6}^2-g_2s_{10} \right)
    \ ,
    \\
    s_{5}^{2} & := s_{10}
    \ ,
    \\
    s_{15}^2
    & :=
    5^5s_{10}^3
    -
    \frac{5^3g_{2}^2s_6s_{10}^2}{2}
    +
    \frac{g_{2}^5s_{10}^2}{2^{4}}
    +
    \frac{3^2\cdot5^2g_2s_{6}^3s_{10}}{2}
    -
    \frac{g_{2}^4s_{6}^2s_{10}}{2^{3}}
    -
    2\cdot3^3s_{6}^5
    +
    \frac{g_{2}^{3}s_{6}^4}{2^{4}}
    \ .
\end{align*}
Finally, the Gundlach--M\"uller invariants for \(\QQ(\sqrt{5})\) 
are 
\[
    J_1 := s_6/g_{2}^3
    \ ,
    \quad
    J_2 := g_{2}^5/s_{5}^2
    \ ,
    \quad 
    \text{and}
    \quad
    J_3 := s_{5}^3/s_{15}
    \ .
\]
The Hilbert modular polynomials for \(\QQ(\sqrt{5})\)
are too large to reproduce here,
but they can be downloaded from
\url{pub.math.leidenuniv.nl/~martindalecr}.\footnote{
    The polynomials $H_{\mu,3}$ do not appear there, 
    but only \(G_\mu\) is required to apply our results
    in~\S\ref{sec:Atkin-g2}.
}

\begin{proposition}[\protect{\cite[Prop.~4.5]{LauYan11} with correction
    to $\phi^{*}(j_{1})$}]
    \label{pullbacks} 
    For $F = \QQ(\sqrt{5})$, the Igusa invariants pull back to
	\begin{align*}
        \phi^*(j_1)
        & =
        4J_2(3J_{1}^2J_2-2)^5
        \ ,
        \\
        \phi^*(j_2)
        & =
        \frac{1}{2}J_2(3J_{1}^2J_2-2)^3
        \ ,
        \\
        \phi^*(j_3)
        & =
        2^{-3}J_2(2J_{1}^2J_2-2)^2(4J_{1}^2J_2+2^5\cdot3^2J_1-3)
        \ .
	\end{align*}
\end{proposition}

For our computations,
we want to write $J_{1}$, $J_{2}$ and $J_{3}$ 
in terms of the Siegel modular forms $\psi_{4}$, $\psi_{6}$, $\chi_{10}$ and $\chi_{12}$.
(For a canonical way of writing 
$J_1$, $J_2$ and $J_3$ in terms of Igusa--Clebsch invariants,
we refer to \cite[Example 2.5.4]{Mar16}.)

\begin{proposition}[\protect{\cite[Example 2.5.4]{Mar16}}]
    \label{invformulae}
    For $F = \QQ(\sqrt{5})$, we have
    \begin{align*}
        J_{2} &= \phi^{*} \big( (\psi_{4}\psi_{6}/\chi_{10}-3^{5}2^{12})(-2-2(\psi_{6}^{2}-2^{12}3^{6}\chi_{12})/\psi_{4}^{3})^{-1} \big)
        \ ,
        \\
        J_{1} &= 3^{2}2^{5}J_{2}^{-1} + \phi^{*} \big( 2^{-6}3^{-3}(1-(\psi_{6}^2-2^{12}3^{6}\chi_{12})/\psi_{4}^{3}) \big)
        \ ,
        \\
        J_{3}^{2} &= 5^{5} - 2^{-1}5^{3}J_{1}J_{2} + 2^{-4}J_{2} + 2^{-1}3^{2}5^{2}J_{2}^{2}J_{1}^{3} - 2^{-3}J_{1}^{2}J_{2}^{2} - 2\cdot 3^{3}J_{2}^{3}J_{1}^{5} + 2^{-4}J_{2}^{3}J_{1}^{4}
        \ .
    \end{align*}
    The choice of square root for $J_{3}$ 
    corresponds to the choice of embedding $\iota$.
\end{proposition}

Proposition~\ref{invformulae}
can be used to find RM invariants for curves drawn from families with
known real multiplication,
before factoring specialized Hilbert modular polynomials in those RM invariants
to derive information on Frobenius.
However, it also crystallizes the rationality question
alluded to at the end of~\S\ref{sec:pullback}:
as we see, a set of values of the Hilbert modular forms over \(\FF_q\)
(or, equivalently, a tuple of Igusa or Igusa--Clebsch invariants over
\(\FF_q\))
only determine \(J_1\), \(J_2\), and \(J_3^2\) over \(\FF_q\).

To get \(J_3\), we need to choose a square root of \(J_3^2\);
but \(J_3^2\) is not guaranteed to be a square in \(\FF_q\).
If \(J_3^2\) is not a square in \(\FF_q\),
then we cannot apply Propositions~\ref{prop:61}
or~\ref{prop:factorization}---%
not even if \(J_3\) does not appear unsquared 
in the specialized polynomial~\(G_\mu\).

\section{
    Experimental results
}
\label{sec:experiments}

In order to validate the factorization patterns of Proposition~\ref{prop:factorization},
we ran a series of experiments for \(F = \QQ(\sqrt{5})\), 
using the family of curves~\cite{TauTopVer91}
\[
    \C_{a}: y^2 = x^5 - 5x^3 + 5x + a
\]
whose Jacobians all have real multiplication by \(\mathcal{O}_{\QQ(\sqrt{5})}\).
This family was used in the point-counting records
of~\cite{Gaudry--Kohel--Smith}.
The Igusa--Clebsch invariants of \(\C_{a}\) are
\[
    (A,B,C,D) 
    =
    \left(
        2^5\cdot5^2\cdot7
        ,
        \;
        2^{10}\cdot5^4
        ,
        \;
        {-2^{13}}\cdot5^5\cdot(9a^2-236)
        ,
        \;
        2^{20}\cdot5^5\cdot(a^2-4)^2
    \right)
    \ .
\]

Our experiments treated
\begin{enumerate}
    \item
        the ramified prime \(\ell = 5\), with 
        \(\mu = (5+\sqrt{5})/2\),
        and the 
        modular polynomial \(G_\mu\)
        from
        {\url{pub.math.leidenuniv.nl/~martindalecr}};
    \item
        the split prime \(\ell = 11\), with $\mu = (7+\sqrt{5})/2$,
        and the 
        modular polynomial \(G_\mu\)
        from
        {\url{pub.math.leidenuniv.nl/~martindalecr}}.
\end{enumerate}

We collected statistics on the factorization patterns for 10000
tests. For each test, we chose a random prime $q$ of ten decimal
digits, and we chose $a$ randomly from \(\FF_q\)
subject to the requirement that \(\C_{a}\) be nonsingular,
which is \(a^2 \not= 4\).
We then applied the formul\ae{} of
Eq.~\eqref{eq:psi-from-IgusaClebsch} and Proposition~\ref{invformulae} 
to obtain the RM invariants \(J_2\) and \(J_1\) for the Jacobian of \(\C_{a}\),
as well as the squared invariant \(J_3^2\).

In half the cases on average, $J_3^2$ had a square root in $\F_q$;
in these cases we could obtain $J_3$,
and proceed to factor \(G_\mu(J_1,J_2,J_3,x)\).
The average frequencies of the resulting factorization patterns 
appear in Tables~\ref{tab:facto-patterns-5}
and~\ref{tab:facto-patterns-11}
(here we take the averages over the roughly 5000 tests
where $J_3^2$ has a root in $\F_q$; for the two roots $J_3$ and $-J_3$
in $\F_q$, we always obtained the same factorization pattern). 

\begin{table}[htp]
    \centering
    \begin{tabular}{|c|r|r|}
        \hline
        Factorization pattern, type of \(\mu\) & Found & Percentage \\  
        \hline
        \(\mathcal{O}_{\QQ(\sqrt{5})}\)-Elkies: $(1,1,e,\ldots,e)$ with \(e > 1\) & total 1835 & total 36.8\%\\
        \hline
        $(1,1,4)$       & 1266 & 25.4\% \\
        $(1,1,2,2)$     &  569 & 11.4\% \\
        \hline
        \(\mathcal{O}_{\QQ(\sqrt{5})}\)-Atkin: $(e,\ldots,e)$ with \(e > 1\) & total 2049 & total 41.1\% \\
        \hline
        $(6)$           &  844 & 16.9\% \\
        $(3,3)$         &  794 & 15.9\% \\
        $(2,2,2)$       &  411 &  8.2\% \\
        \hline
        \(\mathcal{O}_{\QQ(\sqrt{5})}\)-Volcanic: $(1,e)$ or $(1,\ldots,1)$ & total 1105 & total 22.1\%\\
        \hline
        $(1,5)$         & 1058 & 21.2\% \\
        $(1,1,1,1,1,1)$ &   47 &  0.9\% \\
        \hline
    \end{tabular}
    \caption{%
        Factorization pattern frequencies for the modular
        polynomial $G_\mu(J_1,J_2,J_3,x)$ 
        for $\mu=(5+\sqrt{5})/2$ of norm \(\ell=5\). 
        The degree of $G_\mu(J_1,J_2,J_3,x)$ in~\(x\) is~6.
        We only factored when $J_3^2$ was a square in $\F_q$,
        which happened in 4989 of the 10000 trials (49.9\%). 
    }
  \label{tab:facto-patterns-5}
\end{table}

\begin{table}[htp]
    \centering
    \begin{tabular}{|c|r|r|}
        \hline
        Factorization pattern, type of \(\mu\) & Found & Percentage \\  
        \hline
        \(\mathcal{O}_{\QQ(\sqrt{5})}\)-Elkies: 
        $(1,1,e,\ldots, e)$ with \(e > 1\) 
        & total 2262 & total 44.7\% 
        \\
        \hline
        $(1,1,5,5)$                 & 1040 & 20.6\% 
        \\
        $(1,1,10)$                  &  994 & 19.7\% 
        \\
        $(1,1,2,2,2,2,2)$           &  228 &  4.5\% 
        \\
        \hline 
        \(\mathcal{O}_{\QQ(\sqrt{5})}\)-Atkin: 
        $(e,\ldots,e)$ with $e>1$ 
        & total 2329 & total 46.1\%
        \\
        \hline
        $(12)$                      &  859 & 17.0\% 
        \\
        $(6,6)$                     &  404 &  8.0\% 
        \\
        $(4,4,4)$                   &  424 &  8.4\% 
        \\
        $(3,3,3,3)$                 &  429 &  8.5\% 
        \\
        $(2,2,2,2,2,2)$             &  213 &  4.2\% 
        \\
        \hline 
        \(\mathcal{O}_{\QQ(\sqrt{5})}\)-volcanic: 
        $(1,e)$ or $(1,\ldots,1)$ & total 466 & total 9.2\% 
        \\
        \hline
        $(1,11)$                    &  461 &  9.1\% 
        \\
        $(1,1,1,1,1,1,1,1,1,1,1,1)$ &    5 &  0.1\% 
        \\
        \hline
    \end{tabular}
    \caption{Factorization pattern frequencies for the modular
    polynomial $G_\mu(J_1,J_2,J_3,x)$
    for $\mu=(7+\sqrt{5})/2$ of norm~\(\ell=11\). 
    The degree of $G_\mu(J_1,J_2,J_3,x)$ in~\(x\) is~12.
    We only factored when $J_3$ was a square in \(\FF_q\),
    which happened in 5057 of the 10000 trials (50.6\%). }
    \label{tab:facto-patterns-11}
\end{table}

According to Proposition \ref{prop:factorization},
we would expect that $1/\ell$ of the time 
\(\mu\) should be \(\mathcal{O}_{\QQ(\sqrt{5})}\)-volcanic,
$(\ell-1)/2\ell$ of the time 
\(\mu\) should be \(\mathcal{O}_{\QQ(\sqrt{5})}\)-Elkies,
and $(\ell-1)/2\ell$ of the time 
\(\mu\) should be \(\mathcal{O}_{\QQ(\sqrt{5})}\)-Atkin.
The summary of our above results in Table~\ref{tab:totals} appears to confirm this.
This gives us considerable confidence that the Hilbert modular polynomials computed in \cite[Chapter 2]{Mar16} are correct.

Finally, we ran the same tests on Milio's modular
polynomial\footnote{%
    Available from \url{https://members.loria.fr/EMilio/modular-polynomials/}
}
$\Phi(\mathfrak{J_1}, \mathfrak{J_2}, X)$ for $\ell=5$ and
$\mu=(5+\sqrt{5})/2$, where $\mathfrak{J_1}=J_2$ and
$\mathfrak{J_2}=J_1J_2$. 
We obtained exactly the same factorization patterns each time $J_3$
was in $\F_q$.

\begin{table}[htp]
    \centering
    \begin{tabular}{|r@{\quad}r|c|c|c|}
        \hline
        & & \multicolumn{3}{c|}{Prime type frequencies for \(\mu\)} 
        \\
        \cline{3-5}
        &
        & \(\mathcal{O}_{\QQ(\sqrt{5})}\)-volcanic
        & \(\mathcal{O}_{\QQ(\sqrt{5})}\)-Elkies 
        & \(\mathcal{O}_{\QQ(\sqrt{5})}\)-Atkin
        \\
        \hline
        \multirow{2}{*}{\(\mu = \frac{5 - \sqrt{5}}{2}\)} 
        & Theory & $20.0\%$ & $40.0\%$ & $40.0\%$
        \\
        & Experiments & $22.1\%$ & $36.8\%$& $41.1\%$
        \\
        \hline
        \multirow{2}{*}{\(\mu = \frac{7 + \sqrt{5}}{2}\)}
        & Theory & $9.1\%$ & $45.5\%$ & $45.5\%$
        \\
        & Experiments & $9.2\%$ &$44.7\%$ &$46.1\%$
        \\
        \hline
    \end{tabular}
    \caption{%
        Experimental evidence supporting the correctness of
        Martindale's Hilbert modular polynomials.
    }
    \label{tab:totals}
\end{table}

\begin{acknowledgement}
    This article reports on work carried out at the workshop 
    \emph{Algebraic Geometry for Coding Theory and Cryptography}
    at the Institute for Pure and Applied Mathematics (IPAM),
    University of California, Los Angeles,
    February 22--26, 2016. 
    The authors thank IPAM for its generous support.
    Chloe Martindale was supported by an ALGANT-doc scholarship in
    association with Universiteit Leiden and Universit\'e de Bordeaux.
    Maike Massierer was supported by the Australian Research Council (DP150101689).
\end{acknowledgement}

\FloatBarrier

\bibliographystyle{abbrv}
\bibliography{biblio}

\end{document}